\documentclass{amsart}

\usepackage[numbers]{natbib}
\usepackage[colorlinks]{hyperref}

\usepackage{amsmath,amsthm,amssymb}
\usepackage{enumerate}
\usepackage{graphicx}
\usepackage{float}

\theoremstyle{plain}
\newtheorem{thm}{Theorem}[section]
\newtheorem{cor}[thm]{Corollary}
\newtheorem{lem}[thm]{Lemma}
\newtheorem{prop}[thm]{Proposition}

\theoremstyle{definition}
\newtheorem{defn}[thm]{Definition}
\theoremstyle{remark}
\newtheorem{rem}[thm]{Remark}

\usepackage{tikz}
\usetikzlibrary{arrows,shapes,trees}
\usepackage{graphicx}
\usepackage{hyperref}
\usepackage{amssymb}
\usepackage{amsfonts}
\usepackage{amsthm}
\usepackage{amsmath}
\usepackage{epstopdf}
\usepackage{mathrsfs}

\newcommand{\bH}{\mathbb H}
\newcommand{\N}{\mathbb N}

\newcommand{\C}{\mathbb C}

\newcommand{\R}{\mathbb R}

\newcommand{\E}{\mathcal{E}}

\newcommand{\e}{\varepsilon}

\newcommand{\wt}{\widetilde}

\renewcommand{\l}{\ell}
\newcommand{\la}{\lambda}

\newcommand{\p}{\partial}

\title{Rigidity of equality of Lyapunov exponents for geodesic flows}
\author{Clark Butler}
\thanks{This material is based upon work supported by the National Science Foundation Graduate Research Fellowship under Grant No. \# DGE-1144082. }
\begin{document}

\begin{abstract}
We study the relationship between the Lyapunov exponents of the geodesic flow of a closed negatively curved manifold and the geometry of the manifold. We show that if each periodic orbit  of the geodesic flow has exactly one Lyapunov exponent on the unstable bundle then the manifold has constant negative curvature. We also show under a curvature pinching condition that equality of all Lyapunov exponents with respect to volume on the unstable bundle also implies that the manifold has constant negative curvature. We then study the degree to which one can emulate these rigidity theorems for the hyperbolic spaces of nonconstant negative curvature when the Lyapunov exponents with respect to volume match those of the appropriate symmetric space and obtain rigidity results under additional technical assumptions. The proofs use new results from hyperbolic dynamics including the nonlinear invariance principle of Avila and Viana and the approximation of Lyapunov exponents of invariant measures by Lyapunov exponents associated to periodic orbits which was developed by Kalinin in his proof of the Livsic theorem for matrix cocycles. We also employ rigidity results of Capogna and Pansu on quasiconformal mappings of certain nilpotent Lie groups. 
\end{abstract}
\maketitle
\begin{section}{Introduction}
A central question in geometric rigidity theory is the following: Suppose that a negatively curved Riemannian manifold $M$ shares some property $P$ with a negatively curved locally symmetric space $N$. Is $M$ isometric to $N$? The most famous example of such a rigidity theorem is the Mostow rigidity theorem: if $M$ and $N$ are real hyperbolic manifolds with isomorphic fundamental groups, then $M$ and $N$ are isometric. 

We can ask a different, more dynamical rigidity question: Suppose that the geodesic flow of a negatively curved Riemannian manifold $M$ shares some property $P$ with the geodesic flow of a negatively curved locally symmetric space $N$. Is there a $C^{1}$ time-preserving conjugacy between the geodesic flows of $M$ and $N$? A remarkable consequence of the minimal entropy rigidity theorem of Besson, Courtois, and Gallot \cite{BCG1} is that \emph{dynamical rigidity implies geometric rigidity}, in the sense that if there is a $C^{1}$ time preserving conjugacy between the geodesic flows of $M$ and $N$, then $M$ and $N$ are homothetic. Recall that two Riemannian manifolds $(M,d)$ and $(N,\rho)$ with distances $d$ and $\rho$ respectively are \emph{homothetic} if there is a constant $c > 0$ such that $(M,d)$ is isometric to $(N,c\rho)$. This implies that it is possible to characterize the geometry of a locally symmetric space $N$ purely by the dynamics of its geodesic flow. For some examples of the numerous rigidity problems to which this has been applied, see \cite{BL},\cite{BFL}, and \cite{Ham2} as well as the survey articles \cite{BCG2},\cite{Spa}. 

The particular property we will be interested in studying is the conformal geometry and dynamics of the geodesic flow of $M$ acting on the unstable foliation of the unit tangent bundle of $M$. We will show that if this action of the geodesic flow sufficiently resembles that of the corresponding action for the geodesic flow of a negatively curved locally symmetric space then $M$ is actually homothetic to a negatively curved locally symmetric space. Here ``sufficiently resembles" will be measured by the Lyapunov exponents of the geodesic flow and, in the case of negatively curved locally symmetric spaces of nonconstant negative curvature, the regularity of certain dynamically defined subbundles of the unstable bundle. 

Before proceeding further, we fix some notation. Throughout this paper $M$ will denote an $m$-dimensional closed Riemannian manifold of negative curvature with universal cover $\widetilde{M}$. We will always assume $m \geq 3$. We write $SM$ for the unit tangent bundle of $M$. The time-$t$ map of the geodesic flow on $SM$ will be denoted by $g^{t}$. We endow $SM$ with the Sasaki metric, giving $SM$ the structure of a closed Riemannian manifold with norm $\| \cdot \|$ and distance $d$.  We let $\theta$ be the canonical contact 1-form on $SM$ preserved by $g^{t}$. The letter $C$ will be used freely as a multiplicative constant which is independent of whatever parameters are under consideration. 

Since $M$ is negatively curved, $g^{t}$ is an Anosov flow: there is a $Dg^{t}$-invariant splitting $TSM = E^{u} \oplus E^{c} \oplus E^{s}$ with $E^{c}$ tangent to the vector field generating $g^{t}$, and there exist $0 < \nu < 1$, $C > 0$ such that for $v^{u} \in E^{u}$, $v^{s} \in E^{s}$, and $t > 0$,
\[
\|Dg^{t}(v^{s})\| \leq C\nu^{t}\|v^{s}\|, \; \; \|Dg^{-t}(v^{u})\| \leq C\nu^{t}\|v^{u}\|
\]
$E^{u}$, $E^{s}$, and $E^{c}$ are called the \emph{unstable}, \emph{stable}, and \emph{center} subbundles respectively. We write $E^{cu} := E^{u} \oplus E^{c}$ and $E^{cs}:= E^{c} \oplus E^{s}$ for the \emph{center unstable} and \emph{center stable} subbundles respectively. Each of the distributions $E^{u},E^{c},E^{s},E^{cu},E^{cs}$ is uniquely integrable; we denote the corresponding foliations by $W^{*}$, $* = u,c,s,cu,cs$, with $W^{*}(x)$ being the leaf containing $x$. We consider each leaf of these foliations to carry the induced Riemannian metric from the Sasaki metric on $SM$. We let $W^{*}_{r}(x)$ be a ball of radius $r$ centered at $x$ in the leaf $W^{*}(x)$. For a $g^{t}$-invariant subbundle $\E$ of $TSM$, we write $Dg^{t}|\E$ for the restriction of $Dg^{t}$ to $\E$.

Let $\mathbb{K} = \R,\C, \bH,$ or $\mathbb{O}$ be the division algebra of real numbers, complex numbers, quaternions, and octonions respectively. Associated to each of these are the complex, quaternionic, and Cayley hyperbolic spaces $H_{\mathbb{K}}^{k}$ of dimension $k$, $2k$, $4k$, and $8$ respectively. These give the complete list of negatively curved Riemannian  symmetric spaces. We normalize the metrics of these spaces so that they have maximal curvature $-1$. Define $r_{\mathbb{K}} := \dim_{\R}\mathbb{K}$.  

We say that $Dg^{t}|E^{u}$ is \emph{uniformly quasiconformal} if there is a constant $C > 0$ independent of the point $p \in SM$ and $t$ such that for any pair of unit vectors $v,w \in E^{u}_{p}$,
\[
\|Dg^{t}(v)\| \cdot \|Dg^{t}(w)\|^{-1} \leq C.
\]
For closed negatively curved Riemannian manifolds of dimension at least $3$, combined work of Gromov, Kanai, Sullivan and Tukia shows that if the sectional curvatures $K$ of $M$ satisfy $-4 < K \leq -1$ and the action of the geodesic flow on $E^{u}$ is uniformly quasiconformal, then $M$ is homotopy equivalent to a real hyperbolic manifold $N$, and there is a $C^{1}$ time preserving conjugacy between the geodesic flows of $M$ and $N$ (\cite{Gr},\cite{Kan1},\cite{Su},\cite{Tu}). When combined with the minimal entropy rigidity theorem of Besson, Courtois, and Gallot \cite{BCG1}, this implies that $M$ is homothetic to $N$. 

Our first theorem is an improvement of this result. For a periodic point $p$ of $g^{t}$, let $\l(p)$ be the period of $p$. Let $\chi_{1}^{(p)},\dots,\chi_{m-1}^{(p)}$ be the complex eigenvalues of $Dg^{\l(p)}_{p}: E^{u}_{p} \rightarrow E^{u}_{p}$, counted with the multiplicity of their generalized eigenspaces.
\begin{thm}\label{perhyp}
Let $M$ be an $m$-dimensional closed negatively curved Riemannian manifold. Suppose that 
\[
\left|\chi_{i}^{(p)}\right| = \left|\chi_{j}^{(p)}\right|, \; 1 \leq i,j\leq m-1,
\] 
for every periodic point $p$ of the geodesic flow on $SM$. Then $M$ is homothetic to a compact quotient of $H_{\R}^{m}$. 
\end{thm} 
Theorem \ref{perhyp} implies that we can characterize a real hyperbolic manifold by the behavior of a countable collection of linear maps $Dg^{\l(p)}_{p}: E^{u}_{p} \rightarrow E^{u}_{p}$ associated to its geodesic flow. Furthermore, we do not even require these linear maps to be conformal; we only require that for each map $Dg^{\l(p)}_{p}$, all of its eigenvalues have the same absolute value. We've also removed the curvature assumption on $M$. 

We next state a more ergodic theoretic formulation of Theorem \ref{perhyp}. Let $\E$ be a vector bundle over $SM$ carrying a norm $\|\cdot\|$ and let $\pi: \E \rightarrow SM$ be the projection map. A \emph{linear cocycle} over $g^{t}$ is a map $\mathcal{A}: \E \times \R \rightarrow \E$ satisfying for every $t \in \R$, 
\[
\pi(\mathcal{A}(v,t)) = g^{t}(\pi(v)),
\] 
\[
\mathcal{A}(-,t)|_{\E_{x}}: \E_{x} \rightarrow \E_{g^{t}(x)} \; \text{is a linear isomorphism,}
\]
and for any $t,s \in \R$, 
\[
\mathcal{A}(\mathcal{A}(v,s),t) =  \mathcal{A}(v,t+s)
\]
We adopt the notation $A^{t}$ for the map $\mathcal{A}(-,t)$. We will principally be concerned with the linear cocycles obtained by restricting $Dg^{t}$ to invariant subbundles $\E$ of $T(SM)$. For a linear cocycle $A^{t}$ over $g^{t}$ and an ergodic $g^{t}$-invariant measure $\mu$, we define the \emph{extremal Lyapunov exponents} of $A^{t}$ with respect to $\mu$ to be
 \[
\la_{+}(A^{t},\mu) := \inf_{t > 0}\frac{1}{t}\int \log \|A^{t}\| d\mu(x)
\]
\[
\la_{-}(A^{t},\mu)= \sup_{t > 0}\frac{1}{t}\int \log \|A^{-t}\|^{-1} d\mu(x)
\]
We next introduce our hypotheses on the $g^{t}$-invariant measures that we will consider. For each point $x \in SM$ there is an open neighborhood $U_{x}$ of $x$ such that there is a  homeomorphism 
\[
\iota_{x}: W^{cu}_{r}(x) \times W^{s}_{r}(x) \rightarrow U_{x}
\]
for some $r > 0$ given by mapping a pair $(y,z) \in W^{cu}_{r}(x) \times W^{s}_{r}(x)$ to the unique intersection point of $W^{s}(y)$ and $W^{cu}(z)$. A $g^{t}$-invariant probability measure $\mu$ has \emph{local product structure} if for each $x \in SM$ there are measures $\mu_{x}^{cu}$ and $\mu_{x}^{s}$ on $W^{cu}_{r}(x)$ and $W^{s}_{r}(x)$ respectively and a continuous function $\psi_{x}: U_{x} \rightarrow (0,\infty)$ such that 
\[
(\iota_{x})_{*}(d\mu_{x}^{cu} \times d\mu_{x}^{s}) = \psi_{x} d\mu |_{U_{x}}
\]
This says that locally $\mu$ may be expressed as the product of measures on the stable and center-unstable manifolds. Any equilibrium state associated to a H\"older continuous potential for the geodesic flow has local product structure \cite{BR}; these include the Liouville volume on $SM$, the Bowen-Margulis measure of maximal entropy for the geodesic flow, and the harmonic measure corresponding to the hitting probability of Brownian motion inside the universal cover $\widetilde{M}$ of $M$ on the visual boundary $\p \widetilde{M}$.

A negatively curved Riemannian manifold $M$ has relatively $1/4$-pinched sectional curvatures if for each $p \in M$ and each quadruple of tangent vectors $X,Y,W,Z \in T_{p}M$ such that $X$ and $Y$ are linearly independent and $W$ and $Z$ are linearly independent, we have  
\[
K(X,Y) > 4K(W,Z),
\]
where $K(X,Y)$ is the sectional curvature of the plane spanned by $X$ and $Y$. 

\begin{thm}\label{hyp}
Let $M$ be a closed negatively curved Riemannian manifold with relatively $1/4$-pinched sectional curvatures. Let $\mu$ be an ergodic, fully supported $g^{t}$-invariant measure with local product structure. If $\la_{+}(Dg^{t}|E^{u},\mu) = \la_{-}(Dg^{t}|E^{u},\mu)$, then $M$ is homothetic to a compact quotient of $H_{\R}^{m}$.  
\end{thm}

A natural followup question to Theorem \ref{hyp} is the following, 

\textbf{Question}: Is $M$ homothetic to a compact quotient of $H_{\R}^{m}$ if we only assume that $\la_{+}(Dg^{t}|E^{u},\mu) = \la_{-}(Dg^{t}|E^{u},\mu)$? 

A result related to this question was proven by Yue \cite{Yue}, who showed that if $E^{u}$ has no measurable subbundles invariant under $Dg^{t}$ then $M$ is homothetic to a real hyperbolic manifold, without any pinching assumptions on the curvature of $M$. However the proof presented in \cite{Yue} appears to have a significant gap which is discussed in Remark \ref{error} at the end of Section \ref{firstthm}. In a recent preprint \cite{Bu} we corrected this gap to recover Yue's original result. 

Our final two theorems are extensions of Theorem \ref{hyp} to the negatively curved symmetric spaces of nonconstant negative curvature. We first recall the structure of the geodesic flow on $H^{m/r_{\mathbb{K}}}_{\mathbb{K}}$ for $\mathbb{K} = \C, \bH,\mathbb{O}$. The horospheres in $H^{m/r_{\mathbb{K}}}_{\mathbb{K}}$ naturally have the structure of an $(m-1)$-dimensional 2-step  nilpotent Lie group $G_{\mathbb{K}}^{m}$. We fix a left-invariant $(m-r_{\mathbb{K}})$-dimensional distribution $\mathcal{T}_{\mathbb{K}}^{m}$ on $G_{\mathbb{K}}^{m}$ which is transverse to the tangent distribution of the center (as a Lie group) of $G_{\mathbb{K}}^{m}$. We denote the tangent distribution of the center of $G_{\mathbb{K}}^{m}$ by $\mathcal{V}^{m}_{\mathbb{K}}$. 

When we identify the horospheres of $H^{m/r_{\mathbb{K}}}_{\mathbb{K}}$ with $G_{\mathbb{K}}^{m}$ the geodesic flow $g^{t}_{\mathbb{K}}$ of the symmetric metric $d_{\mathbb{K}}$ acts by an expanding automorphism of $G_{\mathbb{K}}^{m}$ which leaves the splitting $TG_{\mathbb{K}}^{m} = \mathcal{T}_{\mathbb{K}}^{m} \oplus \mathcal{V}^{m}_{\mathbb{K}}$ invariant, expands $\mathcal{T}_{\mathbb{K}}^{m}$ by a factor of $e^{t}$, and expands $\mathcal{V}^{m}_{\mathbb{K}}$ by a factor of $e^{2t}$. 

We next need an additional dynamical definition: a \emph{dominated splitting} for a linear cocycle $\mathcal{A}: \E \rightarrow \E$ over $g^{t}$ is an $\mathcal{A}$-invariant direct sum splitting $\E = \E^{1} \oplus \E^{2}$ such that there is some norm $\| \cdot \|$ on $\E$, some $C > 0$ and some $0 < \la < 1$ satisfying for every $p \in SM$, 
\[
\|A^{t}_{p}|\E^{2}_{p}\| \leq C\la^{t}\|A^{-t}_{g^{t}p}|\E^{1}_{p}\|^{-1}, \; t > 0
\]
Dominated splittings are stable under $C^{0}$-perturbations of $\mathcal{A}$ \cite{Shu}. As we saw in the previous paragraph, the geodesic flow of the symmetric metric $d_{\mathbb{K}}$ on $\mathbb{H}_{\mathbb{K}}^{m/r_{\mathbb{K}}}$ admits a dominated splitting of the unstable bundle given in horospherical coordinates by $TG_{\mathbb{K}}^{m} = \mathcal{T}_{\mathbb{K}}^{m} \oplus \mathcal{V}^{m}_{\mathbb{K}}$. Hence if the metric $d$ on a cocompact quotient $M = \mathbb{H}_{\mathbb{K}}^{m/r_{\mathbb{K}}}/ \Gamma$ is $C^{2}$ close to the symmetric metric $d_{\mathbb{K}}$ then the geodesic flow $g^{t}$ of $d$ also admits a dominated splitting on the unstable bundle, $E^{u} = H^{u} \oplus V^{u}$, where $H^{u}$ is the weaker of the two expanding directions. 

The Lie group $G_{\mathbb{K}}^{m}$ carries a natural left-invariant subriemannian metric which is the Carnot-Caratheodory metric defined by the distribution $\mathcal{T}^{m}_{\mathbb{K}}$. With respect to this metric the expanding automorphism $g^{t}_{\mathbb{K}}$ is conformal. Conformal mappings with respect to this metric exhibit strong rigidity properties much like the conformal maps of $\R^{k}$ for $k \geq 3$; all conformal maps from one domain of $G_{\mathbb{K}}^{m}$ to another are the restrictions of a projective automorphism of the group.  For further discussion and details we refer to\cite{P2}.

In order to apply the rigidity of conformal maps of $G_{\mathbb{K}}^{m}$ to the study of the rigidity of the geodesic flow $g^{t}$ of metrics close to $d_{\mathbb{K}}$, there are two requirements: we need that $Dg^{t}| H^{u}$ is conformal we also need charts which identify $H^{u}$ with the left-invariant distribution $\mathcal{T}^{m}_{\mathbb{K}}$ in $G_{\mathbb{K}}^{m}$. Our next theorem states that these two requirements are sufficient to obtain a rigidity theorem. As in Theorem \ref{hyp}, it's sufficient to require only that the extremal Lyapunov exponents of $Dg^{t}|H^{u}$ with respect to a sufficiently structured measure $\mu$ are equal in order to guarantee conformality of $Dg^{t}|H^{u}$. 

\begin{thm}\label{weakchyp}
Let $M = \mathbb{H}_{\mathbb{K}}^{m/r_{\mathbb{K}}}/ \Gamma$ be a cocompact quotient of a symmetric space of nonconstant negative curvature with symmetric metric $d_{\mathbb{K}}$. There is a $C^{2}$ open neighborhood $\mathcal{U}$ of $d_{\mathbb{K}}$ in the space of Riemannian metrics on $M$ such that if $d \in \mathcal{U}$ has geodesic flow $g^{t}$ which satisfies 
\begin{enumerate}
\item  $\la_{+}(Dg^{t}|H^{u},\mu) = \la_{-}(Dg^{t}|H^{u},\mu)$ for some ergodic, fully supported $g^{t}$-invariant measure $\mu$ with local product structure.
\item There is an $r > 0$ such that for every $p \in SM$ there is a $C^{1}$ embedding $\Psi_{p}: W^{u}_{r}(p) \rightarrow G_{\mathbb{K}}$ with $D\Psi_{p}(H_{q}^{u}) = (\mathcal{T}_{\mathbb{K}}^{m})_{\Psi_{p}(q)}$. 
\end{enumerate}
Then $d$ is homothetic to $d_{\mathbb{K}}$. 
\end{thm}

It is possible to give an explicit description of the set $\mathcal{U}$ in terms of norm bounds on $Dg^{t}|H^{u}$ and $Dg^{t}|V^{u}$; this is discussed in detail at the beginning of Section \ref{thirdthm}. A natural question which is important for potential applications is whether Assumption (2) of Theorem \ref{weakchyp} follows from Assumption (1). We are able to show in the case of complex hyperbolic space that if one assumes that the bundle $H^{u}$ has sufficiently high regularity along the unstable foliation $W^{u}$ then Assumption (2) follows from Assumption (1).  

\begin{thm}\label{regchyp}
Let $M = \mathbb{H}_{\C}^{m/2}/ \Gamma$ be a cocompact quotient complex hyperbolic space with symmetric metric $d_{\C}$. There is a $C^{2}$ open neighborhood $\mathcal{U}$ of $d_{\C}$ in the space of Riemannian metrics on $M$ such that if $d \in \mathcal{U}$ has geodesic flow $g^{t}$ which satisfies 
\begin{enumerate}
\item  $\la_{+}(Dg^{t}|H^{u},\mu) = \la_{-}(Dg^{t}|H^{u},\mu)$ for some ergodic, fully supported $g^{t}$-invariant measure $\mu$ with local product structure.
\item The restriction of $H^{u}$ to unstable leaves $W^{u}$ is $C^{1}$ with H\"older continuous derivative.  
\end{enumerate}
Then Assumption (2) of Theorem \ref{weakchyp} holds and consequently $d$ is homothetic to $d_{\C}$. 
\end{thm}

The high regularity assumption on $H^{u}$ is likely too strong; on the open neighborhood $\mathcal{U}$ of $d_{\mathbb{C}}$ given in Theorem \ref{regchyp}, $H^{u}$ will typically be no better than H\"older continuous along $W^{u}$ with exponent $1/2-\e$ for some $\e > 0$. We expect that Theorem \ref{regchyp} also holds for perturbations of the geodesic flow quaternionic hyperbolic and octonionic hyperbolic manifolds. Lastly we make the important remark that verifying Assumption (2) of Theorem \ref{weakchyp} does not require proving higher regularity of $H^{u}$ as assumed in Theorem \ref{regchyp}. 

The proofs of these theorems make use of two powerful tools recently developed in smooth dynamics. The first is the method of approximation of Lyapunov exponents of invariant measures over a system by Lyapunov exponents of periodic points developed by Kalinin in his recent solution of the Livsic problem for $GL(n,\R)$-cocycles over hyperbolic systems\cite{Kal}. We use this to transfer information about the periodic exponents of $g^{t}$ to exponents of any invariant measure for $g^{t}$. 

The second is a far-reaching nonlinear generalization of Furstenberg's theorem on nonvanishing Lyapunov exponents for random $GL(n,\R)$-cocycles which characterizes when the Lyapunov exponents of a cocycle over a partially hyperbolic system vanish under suitable hypotheses. Inspired by an alternative proof by Ledrappier \cite{Led} of Furstenberg's theorem, Avila and Viana proved a nonlinear generalization \cite{AV}, and then later with Santamaria showed how this nonlinear generalization could be applied to cocycles over partially hyperbolic systems \cite{ASV}. We apply a further distillation of this tool by Kalinin and Sadovskaya in \cite{KS} which is adapted to the study of cocycles which are close to being conformal. They have applied this to the study of linear cocycles with uniformly quasiconformal behavior and asymptotically conformal Anosov diffeomorphisms \cite{KS2},\cite{KS3}.

In Section \ref{amenred} we adapt the main results of Kalinin and Sadovskaya \cite{KS} regarding conformal structures for linear cocycles to our setting. We also review the concepts of fiber bunching and stable and unstable holonomies from partially hyperbolic dynamics. In Section \ref{firstthm} we prove Theorem \ref{perhyp} and Theorem \ref{hyp}. In Section \ref{horsub} we analyze the case of a dominated splitting $E^{u} = H^{u} \oplus V^{u}$ and develop the dynamical tools needed for the proofs of the remaining results. In Section \ref{thirdthm} we use these tools to prove Theorem \ref{weakchyp}. In Section \ref{regproof} we prove Theorem \ref{regchyp}. 

\textbf{Acknowledgments}: We thank Amie Wilkinson for numerous discussions which improved the paper greatly, as well as continued guidance and support. We thank Boris Kalinin for bringing to our attention the gap in \cite{Yue} discussed in Remark \ref{error}. We thank Ursula Hamenst\"adt for providing the closing argument in the last paragraph of the proof of Theorem \ref{weakchyp}. We are grateful to Alex Eskin for pointing out an error in the proof of Lemma \ref{amen} in an earlier draft of the paper. 
\end{section}

\begin{section}{Background on Linear Cocyles over Anosov Flows}\label{amenred}
For this section we take $g^{t}$ to be an Anosov flow on a Riemannian manifold $X$. We will specialize when necessary to the case that $X = SM$ is the unit tangent bundle of a Riemannian manifold and $g^{t}$ is the geodesic flow. 

A $d$-dimensional vector bundle $\pi: \E \rightarrow X$ is $\beta$-H\"older continuous if there is an open cover of $X$ by open sets $U_{i}$ which admit linear trivializations $\varphi_{i}: \pi^{-1}(U_{i}) \rightarrow U_{i} \times \R^{d}$ such that the transition maps $\varphi_{i} \circ \varphi_{j}^{-1}$ are $\beta$-H\"older continuous with respect to the Riemannian metric on $X$ and the Euclidean metric on $\R^{d}$.  A cocycle $A^{t}: \E \rightarrow \E$ over $g^{t}$ is $\beta$-H\"older if $\E$ is a $\beta$-H\"older vector bundle and $A^{t}$ is $\beta$-H\"older in any $\beta$-H\"older family of trivializations. A standard example of a linear cocycle defined on a H\"older continuous vector bundle is the restriction of $Dg^{t}$ to the unstable bundle $E^{u}$. 

\begin{subsection}{Holonomies for linear cocycles}
We begin with a definition, 

\begin{defn}A $\beta$-H\"older continuous cocycle $A^{t}$ is $\alpha$-\emph{fiber bunched} if $\alpha \leq \beta$ and there is some $T > 0$ such that 
\[
\|A^{t}_{p}\| \cdot \|A^{-t}_{p}\| \cdot \|Dg^{t}_{p}|E^{s}\|^{\alpha} < 1, \; p \in SM, \; t \geq T
\]
\[
\|A^{t}_{p}\| \cdot \|A^{-t}_{p}\| \cdot \|Dg^{-t}_{p}|E^{u}\|^{\alpha} < 1, \; p \in SM, \; t \geq T
\]
\end{defn}

Fiber bunching guarantees the existence of $A^{t}$-equivariant identifications of the fibers of $\E$ along stable and unstable manifolds of $g^{t}$ known as \emph{holonomies}. These identifications are essential for everything that follows in this paper. 

We define here unstable holonomies. Stable holonomies are defined similarly. Stable holonomies could also be defined as the unstable holonomies of the inverse cocycle $A^{-t}$ over $g^{-t}$. The definition below is from \cite{KS}, adapted to the setting of flows. 
\begin{defn}\label{defuhol}
An \emph{unstable holonomy} for a linear cocycle $\mathcal{A}: \E \times \R \rightarrow \E$ over $g^{t}$ is a continuous map $h^{u}:(x,y) \rightarrow h^{u}_{xy}$, where $x \in X$, $y \in W^{u}_{r}(x)$, such that
\begin{enumerate}
\item $h^{u}_{xy}$ is a linear map from $\E_{x}$ to $\E_{y}$;
\item $h^{u}_{xx} = Id$ and $h^{u}_{yz} \circ h^{u}_{xy} = h^{u}_{xz}$;
\item $h^{u}_{xy} = (A^{t}_{y})^{-1} \circ h^{u}_{g^{t}x g^{t}y} \circ A^{t}_{x}$ for every $t \in \R$. 
\end{enumerate}
\end{defn}
The next proposition gives a sufficient condition for the existence of holonomies. For a $\beta$-H\"older vector bundle $\E$, it is always possible to find a $\beta$-H\"older continuous system of linear identifications $I_{xy}: \E_{x} \rightarrow \E_{y}$ with $I_{xx} = Id_{\E_{x}}$ and $d(x,y) \leq r$ for some constant $r > 0$ \cite{KS}. 
\begin{prop}\label{existuhol}
Suppose that $\mathcal{A}$ is $\beta$-H\"older and fiber bunched. Then there is an unstable holonomy $h^{u}$ for $\mathcal{A}$ which satisfies 
\begin{equation}\label{uhol}
\|h^{u}_{xy}-I_{xy}\| \leq Cd(x,y)^{\beta}
\end{equation}
for $x \in SM$, $y \in W^{u}_{r}(x)$, and some $C > 0$. Furthermore, the unstable holonomy satisfying \eqref{uhol} for some $C > 0$ is unique. 
\end{prop}
The proof of Proposition \ref{existuhol} for fiber bunched cocycles over partially hyperbolic maps is given in \cite{KS}; an identical proof works for fiber bunched cocycles over Anosov flows instead. 

Let $r > 0$ be small enough that for every $x \in X$, all of the foliations $W^{*}_{r}$ are trivial on the ball of radius $r$. Given an unstable holonomy $h^{u}$ for a linear cocycle $\mathcal{A}$ over $g^{t}$, we can locally extend it to a \emph{center unstable holonomy} $h^{cu}$ by defining for $y \in W^{cu}_{r}(x)$,
\[
h^{cu}_{xy} = A^{\tau}_{g^{-\tau}_{y}} \circ h^{u}_{x g^{-\tau}y} = h^{u}_{g^{\tau}x y} \circ A^{\tau}_{x}
\]
where $\tau = \tau(x,y)$ is the unique real number such that $g^{-\tau}y \in W^{u}_{r}(x)$. It is easily checked using the properties in Definition \ref{defuhol} that $h^{cu}$ satisfies properties analogous to those of $h^{u}$ on a ball of radius $r$. In particular, for $y \in W^{cu}_{r/2}(x)$ and $z \in W^{cu}_{r/2}(y)$, we have that $h^{cu}_{yz} \circ h^{cu}_{xy} = h^{cu}_{xz}$. We also see that if $y \in W^{cu}_{r}(x)$ and $d(g^{t}y,g^{t}x) \leq r$, then  $h^{cu}_{xy} = (A^{t}_{y})^{-1} \circ h^{cu}_{g^{t}x g^{t}y} \circ A^{t}_{x}$. 

We specialize to the case $X = SM$ and $g^{t}$ a geodesic flow for the rest of this subsection. It is not immediately clear from the formula defining the center unstable holonomies that these extend to be globally defined on a center unstable leaf $W^{cu}(x)$; to prove this we use some of the special structure of the geodesic flow. Since $m \geq 3$, the universal cover of $SM$ is the unit tangent tangent bundle $S\widetilde{M}$ of the universal cover $\widetilde{M}$ of $M$, and $\pi_{1}(SM)$ is canonically isomorphic to $\pi_{1}(M)$ by the projection $SM \rightarrow M$. The foliations $W^{*}$ lift to foliations $\widetilde{W}^{*}$ of $S\widetilde{M}$ which have \emph{global product structure}: for each $x,y,z \in S\widetilde{M}$, the leaves $\widetilde{W}^{c}(x), \widetilde{W}^{u}(y),\widetilde{W}^{s}(z)$ intersect in exactly one point.

Let $\widetilde{\E}$ be the lift of the vector bundle $\E$ to a vector bundle over $S \widetilde{M}$. For two points $x \in S\widetilde{M}$, $y \in \widetilde{W}^{cu}(x)$, the center unstable holonomy map $h^{cu}_{xy}: \widetilde{\E}_{x} \rightarrow \widetilde{\E}_{y}$ is defined by the formula 
\[
h^{cu}_{xy} = \widetilde{A}^{\tau}_{\widetilde{g}^{-\tau}_{y}} \circ h^{u}_{x \widetilde{g}^{-\tau}y} = h^{u}_{\widetilde{g}^{\tau}x y} \circ \widetilde{A}^{\tau}_{x}
\]
where $\tau = \tau(x,y)$ is the unique time $\tau \in \R$ such that $\widetilde{g}^{-\tau}y \in \widetilde{W}^{u}(x)$. It's easy to check that this locally agrees with the previously defined center unstable holonomy, and gives a global extension of $h^{cu}$ satisfying the analogous properties in Definition \ref{defuhol}. 

Let $\p \widetilde{M}$ be the visual boundary of $\widetilde{M}$. This global product structure corresponds to the \emph{Hopf parametrization},
\[
S\widetilde{M} = \R \times \p \widetilde{M} \times \p \widetilde{M} \backslash \Delta(\p \widetilde{M} \times \p \widetilde{M})
\]
given as follows: Fix a basepoint $x \in \widetilde{M}$. Let $v \in S\widetilde{M}$. $v$ is tangent to a geodesic $\gamma_{v}$ which has endpoints $v_{+},v_{-} \in \p \widetilde{M}$, where $v_{+}$ corresponds to the forward endpoint of $\gamma_{v}$, and $v_{-}$ the backward endpoint. Let $x_{v}$ be the orthogonal projection of $x$ onto $\gamma_{v}$, and let $s$ be the distance from $x_{v}$ to $P(v)$, where $P: S\widetilde{M} \rightarrow \widetilde{M}$ is projection. Then the identification is given by $v \rightarrow (s,v_{+},v_{-})$. In this identification, the action of the geodesic flow is given by translation in the $\R$-coordinate. This parametrization of $S\widetilde{M}$ will be important in Sections \ref{firstthm} and \ref{horsub}. 

\end{subsection}

\begin{subsection}{Continuous Amenable Reduction}\label{contamen} 
We now adapt the main results of \cite{KS} to our setting. Let $\E$ be a $d$-dimensional H\"older continuous vector bundle over a Riemannian manifold $X$ with an Anosov flow $g^{t}$.  We let $\mu$ be a fully supported ergodic $g^{t}$-invariant measure with local product structure. For the results in this section we will assume that the stable and unstable distributions $E^{s}$ and $E^{u}$ for $g^{t}$ are not jointly integrable; this is true for the geodesic flow because the geodesic flow is a contact Anosov flow.  

Two Riemannian metrics $\tau$ and $\sigma$ on $\E$ are \emph{conformally equivalent} if there is a function $a: X \rightarrow \R$ such that $\tau_{p} = a(p)\sigma_{p}$. A \emph{conformal structure} on $\E$ is a conformal equivalence class of Riemannian metrics on $\E$. $A^{t}$ transforms a conformal structure by pulling back the associated Riemannian metric. A conformal structure represented by a Riemannian metric $\tau$ is \emph{invariant} under $\mathcal{A}$ if for each $t \in \R$ there is a map $\psi^{t}: X \rightarrow \R$ satisfying
\[
(A^{t})^{*}\tau = \psi^{t} \tau
\]
In this case we say that $\psi^{t}$ is the multiplicative cocycle associated to the invariant conformal structure $\tau$. $\psi^{t}$ satisfies the cocycle property
\[
\psi^{t+s}(p) = \psi^{t}(p)\psi^{s}(g^{t}(p))
\]
for any $t,s \in \R$. 

Two multiplicative cocycles $\psi^{t}$ and $\varphi^{t}$ are \emph{cohomologous} if there is a map $\zeta: X \rightarrow \R$ such that 
\[
\frac{\psi^{t}}{\varphi^{t}} = \frac{\zeta \circ g^{t}}{\zeta}
\]
for every $t \in \R$. 

If a cocycle $\mathcal{A}$ over $X$ admits stable and unstable holonomies, we say that a subbundle $\mathcal{V} \subset \E$ is \emph{holonomy invariant} if for $y \in W^{*}(x)$ we have $h^{*}_{xy}(\mathcal{V}_{x}) = \mathcal{V}_{y}$ for $* = u$ or $s$. Similarly we say that a conformal structure is holonomy invariant if it is invariant under pulling back by stable and unstable holonomies. 
\begin{lem}\label{contsub}
Let $\mathcal{A}$ be a fiber bunched cocycle over an Anosov flow $g^{t}$ for which $E^{u}$ and $E^{s}$ are not jointly integrable. Suppose that
\[
\la_{+}(\mathcal{A},\mu) = \la_{-}(\mathcal{A},\mu)
\] 
Then any measurable $\mathcal{A}$-invariant subbundle $V \subseteq \E$ coincides $\mu$-a.e. with a $\mathcal{A}$-invariant holonomy invariant continuous subbundle. Under the same hypotheses, any $\mathcal{A}$-invariant measurable conformal structure $\tau$ on $\E$ coincides $\mu$-a.e. with a $\mathcal{A}$-invariant holonomy invariant continuous conformal structure.
\end{lem}

\begin{proof}
The cocycle generated by $A^{1}$ is a fiber bunched cocycle over the partially hyperbolic diffeomorphism $g^{1}$. Since $E^{u}$ and $E^{s}$ are not jointly integrable, $g^{1}$ is accessible as a partially hyperbolic diffeomorphism \cite{BPW}. In the first case $V$ is a measurable invariant subbundle for $A^{1}$; in the second case, $\tau$ is an invariant measurable conformal structure for $A^{1}$. Theorem 3.3 and Theorem 3.1 respectively from \cite{KS} then apply to give the desired result. 
\end{proof}

\begin{lem}\label{amen}
Let $\mathcal{A}$ be a fiber bunched cocycle over an Anosov flow $g^{t}$ such that $E^{u}$ and $E^{s}$ are not jointly integrable. Suppose that
\[
\la_{+}(\mathcal{A},\mu) = \la_{-}(\mathcal{A},\mu).
\] 
Then there is a finite cover $\mathcal{X}$ of $X$ and a flag
\[
0 \subsetneq \E^{1} \subsetneq \E^{2} \subsetneq \dots \subsetneq \E^{k} = \widetilde{\E}
\]
of continuous holonomy-invariant subbundles $\E^{i}$ which are invariant under the action of the lifted cocycle $\widetilde{\mathcal{A}}$ on the lifted bundle $\widetilde{E}$ over $\mathcal{X}$. Furthermore the induced action of the cocycle $\widetilde{\mathcal{A}}_{i}$ on $\E^{i}/\E^{i-1}$ preserves a continuous holonomy invariant conformal structure. 
\end{lem}

\begin{proof}
The vector bundle $\E$ admits a measurable trivialization on a set of full $\mu$-measure by Proposition 2.12 in \cite{BP}. Since $\mu$ is fully supported on $X$, this implies that there is a measurable map $P: \E \rightarrow X\times \R^{d}$ commuting with the projections onto $X$ and which is linear on the fibers. $\mathcal{B} = P\mathcal{A}P^{-1}$ is a measurable linear cocycle over $g^{t}$ on the trivial vector bundle $X \times \R^{d}$. We can apply Zimmer's amenable reduction theorem \cite{Zim} for $\R$-cocycles to conclude that there is a measurable map $C:X \rightarrow GL(d,\R)$ such that the cocycle $\mathcal{F} = C\mathcal{B}C^{-1}$ takes values in an amenable subgroup $G$ of $GL(d,\R)$. 

The maximal amenable subgroups of $GL(d,\R)$ are classified in \cite{Moo}. Any such group $G$ contains a finite index subgroup $K$ which is conjugate to a subgroup of a group of the form
\[
H(d_{1},\dots,d_{k}) = \left[\begin{array}{cccc} A_{1} & * & * & *\\
0 & A_{2} & * & * \\
0 & 0 & \ddots & * \\
0 & 0 & 0 & A_{k}
\end{array}\right]
\]
where $\sum_{i=1}^{k}d_{i} = d$ and $A_{i} \in \R \cdot SO(d_{i},\R)$. Thus, by conjugating the cocycle $\mathcal{F}$ if necessary, we may assume that $\mathcal{F}$ takes values in a group $G$ which contains a finite index subgroup $K$ that is contained in one of the groups $H(d_{1},\dots,d_{k})$. Let $G_{*}$ be the stabilizer in $G$ of the flag $V^{1} \subset V^{2} \subset \dots \subset V^{k} = \R^{d}$ corresponding to the group $H(d_{1},\dots,d_{k})$ containing $K$. Thus $V^{j}$ is the span of the first $\sum_{i=1}^{j}d_{i}$ coordinate axes in $\R^{d}$. Let $\l$ be the index of $G_{*}$ in $G$, which is finite since $K$ has finite index in $G$ and $K \subset G_{*}$.

Let $V^{i,j}$, $j = 1, \dots, \l$ be the at most $\l$ distinct images of the subspace $V^{i}$ under the action of $G$. Let $U^{i} = \bigcup_{j=1}^{\l}V^{i,j}$. Then let $\widehat{\E}^{i,j}_{x} = (C \circ P)^{-1}(x) \cdot V^{i,j}$, $\widehat{\mathcal{U}}^{i} = (C \circ P)^{-1}(x)\cdot U^{i}$. The proof of Theorem 3.4 in \cite{KS} shows that if the union of measurable subbundles $\widehat{\mathcal{U}}^{i}$ is invariant under a fiber bunched cocycle with equal extremal exponents over an accessible partially hyperbolic system (which we can take to be the time 1 map $A^{1}$ of the cocycle $\mathcal{A}$ over $g^{1}$), then there is a finite cover $\mathcal{X}$ of $X$ such that the individual subbundles $\widehat{\E}^{i,j}_{x}$ lift to subbundles $\E^{i,j}$ of the lifted bundle $\widetilde{\E}$ over $\mathcal{X}$ which agree $\mu$-a.e. with continuous subbundles which we will also denote $\E^{i,j}$. By construction the lifts $\mathcal{U}^{i}$ are invariant $\mu$-a.e. under the action of the lift $\widetilde{\mathcal{A}}$ of the cocycle $\mathcal{A}$. This is because we constructed these unions of subbundles using amenable reduction over the $\R$ action given by $\mathcal{A}$, and under our measurable trivialization $\mathcal{A}$ takes values in the group $G$. Since $\mathcal{A}$ is continuous and the lifts $\mathcal{U}^{i}$ are continuous after modification on a $\mu$-null set, we conclude that each $\mathcal{U}^{i}$ is everywhere invariant under $\mathcal{A}$.

For each $i \in \{1,\dots,k\}$, $x \in \mathcal{X}$, $t \in \R$, and $j \in \{1,\dots,\l\}$, there is thus an integer $S_{i}(x,t,j)$ such that $A^{t}(\E_{x}^{i,j}) = \E_{g^{t}x}^{i,S(x,t,j)}$. For a fixed $i$ and $j$, $S_{i}(x,t,j)$ depends continuously on $x$ and $t$ since both $\widetilde{A}^{t}$ and all of the subbundles $\E^{i,j}$ are continuous. Since for a fixed $i$ and $j$ we have that $S_{i}(x,t,j)$ is continuous, integer valued, and has connected domain $\mathcal{X} \times \R$, we conclude that $S_{i}(x,t,j):=S_{i}(j)$ is constant in $x$ and $t$. Furthermore, since $S_{i}(x,0,j) = j$, we conclude that $S_{i}(j) = j$. Hence all of the subbundles $\E^{i,j}$ are invariant under $\widetilde{\mathcal{A}}$ as well. In particular $\mathcal{A}$ preserves the flag $\E^{1} \subset \dots \subset \E^{k}$ which arises as the continuous extension of the lift of the flag coming from the standard flag $V^{1} \subset V^{2} \subset \dots \subset V^{k}$. 

To prove the second claim, note that for any $r \geq 1$, the induced action of the cocycle $\mathcal{F}$ on $V^{\sum_{i=1}^{r}d_{i}}/V^{\sum_{i=1}^{r-1}d_{i}} = \R^{d_{r}}$ preserves the standard Euclidean conformal structure on $\R^{d_{r}}$. This immediately implies that $\widetilde{\mathcal{A}}$ preserves a measurable conformal structure on the corresponding quotient bundle $\E^{j}/\E^{j-1}$. By Lemma \ref{contsub}, this measurable conformal structure coincides $\mu$-a.e. with a holonomy invariant continuous conformal structure. 
\end{proof}

\begin{lem}\label{confcover}
Suppose that there is a finite cover $\mathcal{X}$ of $X$ such that the lifted cocycle $\widetilde{\mathcal{A}}$ on the lifted bundle $\widetilde{\mathcal{E}}$ preserves a continuous holonomy-invariant conformal structure. Then $\mathcal{A}$ also preserves a continuous holonomy-invariant conformal structure. 
\end{lem}

\begin{proof}
Let $\widetilde{\mathcal{C}}_{x}$ be the space of conformal structures on the vector space $\widetilde{\E}_{x}$. $\widetilde{\mathcal{C}}_{x}$ can be identified with the Riemannian symmetric space $SL(d,\R)/SO(d,\R)$ and in fact carries a canonical Riemannian metric of nonpositive curvature for which the induced map $\widetilde{\mathcal{C}}_{x} \rightarrow \widetilde{\mathcal{C}}_{g^{t}x}$ over the cocycle $\widetilde{\mathcal{A}}$ is an isometry \cite{KS}. In particular, for compact subsets $K \subset \widetilde{\mathcal{C}}_{x}$ there is a natural barycenter map $K \rightarrow \text{bar}(K)$ mapping $K$ to its center of mass. 

Let $\tau$ be the continuous holonomy-invariant conformal structure preserved by $\widetilde{\mathcal{A}}$. Let $H$ be the group of covering transformations for $\mathcal{X}$ over $X$, which also acts as the group of covering transformations for $\widetilde{\E}$ over $\E$. Let $K_{x} = \bigcup_{\rho \in H}\{\rho \cdot \tau_{\rho^{-1}(x)}\} \subset \widetilde{C}_{x}$. The collection of compact subsets $K_{x}$ depends continuously on $x$, is holonomy-invariant, and is invariant under $\mathcal{A}$. Hence all of the same is true of the family of barycenters $\sigma_{x}:= \text{bar}(K_{x})$. We thus get a conformal structure $\sigma$ that is continuous, holonomy-invariant, invariant under $\widetilde{\mathcal{A}}$, and also invariant under the action of the deck group $H$. $\sigma$ then descends to the desired conformal structure on $\E$. 
\end{proof}

In subsequent sections we will use Lemmas \ref{amen} and \ref{confcover} together to construct invariant conformal structures for our cocycles of interest. We will first use Lemma \ref{amen} to construct an invariant flag on a finite cover, then we will show this flag must be trivial, then lastly we will use Lemma \ref{confcover} to push the invariant conformal structure back down to our original bundle. 

\begin{rem}
The assumption that the stable and unstable distributions $E^{u}$ and $E^{s}$ of $g^{t}$ are not jointly integrable is likely unnecessary in Lemmas \ref{contsub} and \ref{amen}. Different arguments are needed in the case that $E^{u}$ and $E^{s}$ are jointly integrable however, as one cannot use accessibility of the time one map $g^{1}$ in this case. 
\end{rem}

\end{subsection}

\begin{subsection}{Lyapunov exponents and periodic approximation} For a cocycle $\mathcal{A}: \E \times \R \rightarrow \E$ over $g^{t}$ and an ergodic $g^{t}$-invariant measure $\mu$, the multiplicative ergodic theorem \cite{BP} implies that there is a $g^{t}$-invariant subset $\Lambda \subset X$ with $\mu(\Lambda) = 1$ such that over $\Lambda$ there is a measurable $g^{t}$-invariant splitting 
\[
\E = \E^{1} \oplus \E^{2} \oplus \dots \E^{k}
\]
and numbers $\la_{1} < \la_{2} < \dots < \la_{k}$ such that 
\[
\lim_{t \rightarrow \infty} \frac{1}{t}\log \|A^{t}(v)\| = \la_{i}, \; v \in \E^{i}.
\]
The numbers $\la_{i}$ are the \emph{Lyapunov exponents} of $\mathcal{A}$. The extremal Lyapunov exponents $\la_{+}$ and $\la_{-}$ of $\mathcal{A}$ with respect to $\mu$ correspond to the top and bottom exponents $\la_{k}$ and $\la_{1}$ respectively. 

For each periodic point $p$, we let $\mu_{p}$ denote the unique $g^{t}$-invariant probability measure supported on the orbit of $p$, which may be obtained as the normalized pushforward of Lebesgue measure on $\R$ by the map $t \rightarrow g^{t}(p)$. The following theorem of Kalinin enables us to approximate the Lyapunov exponents of any $g^{t}$-invariant measure by the Lyapunov exponents of measures concentrated on a periodic orbit. This theorem is the essential new tool needed for the proof of the Livsic theorem in the case of matrix cocycles. The fact that $g^{t}$ satisfies the closing property necessary in the hypothesis of the theorem as stated in \cite{Kal} is the well known Anosov closing lemma for flows which can be found in Chapter 18 of \cite{HK}. The statement of Theorem \ref{perapprox} in \cite{Kal} assumes that the vector bundle $\E$ over $X$ is trivial, but as remarked by Kalinin in the paper, this hypothesis is easily removed since the proof of the theorem only uses local comparisons between fibers. When we say that the Lyapunov exponents are counted with multiplicity, we mean that each exponent appears a number of times equal to the dimension $\dim \E^{i}$ of its corresponding measurable invariant subbundle. 

\begin{thm}\label{perapprox}[\cite{Kal}]
Let $\E$ be a $d$-dimensional H\"older continuous vector bundle over $X$, and $\mathcal{A}$ a H\"older continuous cocycle on $\E$ over $g^{t}$. Let $\mu$ be an ergodic $g^{t}$-invariant measure, and let $\la_{1} \leq \la_{2} \leq \dots \leq \la_{d}$ be the Lyapunov exponents of $\mathcal{A}$ with respect to $\mu$, counted with multiplicity. Then for every $\e > 0$, there is a periodic point $p$ of $g^{t}$ such that the Lyapunov exponents $\la_{1}^{(p)} \leq \la_{2}^{(p)} \leq \dots \leq \la_{d}^{(p)}$ of $\mathcal{A}$ with respect to $\mu_{p}$ satisfy 
\[
|\la_{i} - \la_{i}^{(p)}| < \e
\]
for each $1 \leq i \leq d$. 
\end{thm}

For a periodic point $p$ there is a simple relationship between the Lyapunov exponents $\la_{i}^{(p)}$ associated to $\mu_{p}$ and the complex eigenvalues $\chi_{i}^{(p)}$ of the map $A^{\l(p)}_{p}: \E_{p} \rightarrow \E_{p}$. Let $\E_{p} = \E^{1}_{p}  \oplus \dots \oplus \E^{k}_{p}$ be the direct sum decomposition of $\E_{p}$ from the multiplicative ergodic theorem and let $\E_{p} = \mathcal{V}^{1} \oplus \dots \mathcal{V}^{r}$ be the primary decomposition of the linear transformation $A^{\l(p)}_{p}: \E_{p} \rightarrow \E_{p}$, where each $\mathcal{V}^{i}$ corresponds to an irreducible factor of the minimal polynomial $A^{\l(p)}_{p}$. An easy linear algebra exercise shows that the primary decomposition is subordinate to the Oseledec decomposition, i.e., for each $1 \leq i \leq k$, 
\[
\E^{i} = \mathcal{V}^{i_{1}} \oplus \dots \oplus \mathcal{V}^{i_{n}}
\]
for some integers $1 \leq i_{1},\dots,i_{n} \leq r$. Furthermore, we have the relationship 
\[
\frac{1}{\l(p)}\log |\chi_{i_{j}}^{(p)}| = \la_{i}^{(p)}, \; 1 \leq j \leq n
\] 
for the real eigenvalues (or conjugate pairs of complex eigenvalues) corresponding to the subspaces $\mathcal{V}^{i_{j}}$. Thus the Lyapunov exponents of $\mu_{p}$ are given by the logarithms of the absolute values of the eigenvalues of $A^{\l(p)}_{p}$, normalized by the period of $p$. 
\end{subsection}
\end{section}

\begin{section}{Proof of Theorems 1.1 and 1.3}\label{firstthm}
We are now ready to prove Theorem \ref{perhyp} and Theorem \ref{hyp}. A subbundle $\mathcal{V} \subset \E$ is \emph{proper} if $0 < \dim \mathcal{V} < \dim \E$. For $0 < \alpha < 2$ we say that $g^{t}$ is $\alpha$-bunched if there is some $T > 0$ such that for $t \geq T$
\[
\|Dg^{t}_{p}|E^{u}_{p}\|^{\alpha} \cdot \|Dg^{t}_{p}|E^{s}_{p}\| \cdot \|(Dg^{t}_{p})^{-1}|E^{u}_{g^{t}(p)}\| < 1, \; t \geq T, p \in SM
\]
For an in-depth discussion of the relationship between $\alpha$-bunching and the regularity of the Anosov splitting $T(SM) = E^{u} \oplus E^{c} \oplus E^{s}$, see \cite{Has2}.

\begin{lem}\label{hypirr}
Let $g^{t}$ be the geodesic flow on the unit tangent bundle of a closed negatively curved manifold. Suppose that $g^{t}$ is $1$-bunched and that there is a $g^{t}$-invariant  fully supported ergodic probability measure $\mu$ with local product structure such that $\la_{+}(Dg^{t}|E^{u},\mu) = \la_{-}(Dg^{t}|E^{u},\mu)$. Then $E^{u}$ has no proper measurable $g^{t}$-invariant subbundles. 
\end{lem}

\begin{proof}
Let $V \subset E^{u}$ be a $k$-dimensional measurable invariant subbundle. Since $g^{t}$ is $1$-bunched the Anosov splitting of $g^{t}$ is $C^{1}$ \cite{Has2} and thus $E^{u}$ is a $C^{1}$-subbundle of $T(SM)$. $1$-bunching of $g^{t}$ also implies that the cocycle $Dg^{t}|E^{u}$ is fiber bunched. By Theorem \ref{contsub}, $V$ thus coincides $\mu$-a.e. with a continuous holonomy invariant subbundle which we will still denote by $V$.  

We now describe an alternative realization of the holonomies for $Dg^{t}|E^{u}$. Recall that $\theta$ denotes the the invariant contact form for $g^{t}$. Since the Anosov splitting of $g^{t}$ is $C^{1}$, and $g^{t}$ preserves  $\theta$, there is a unique $g^{t}$-invariant connection $\nabla$ on $SM$ such that the torsion of $\nabla$ is given by $\theta \otimes \dot{g}$, where $\dot{g}$ is the vector field generating $g^{t}$ on $SM$. This connection is called the \emph{Kanai connection} and was constructed for contact Anosov flows with $C^{1}$ Anosov splitting in \cite{Kan1}. 

In Lemma 1.1 of \cite{Kan1}, it is shown that the unstable foliation $W^{u}$ is totally geodesic for $\nabla$, and that $\nabla$ is $C^{1}$ when restricted to the leaves of the unstable foliation, and further that $\nabla$ is flat when restricted to $W^{u}$ leaves. The parallel transport induced by $\nabla$ on unstable leaves is thus a $C^{1}$ unstable holonomy for $g^{t}$. From the uniqueness clause of Proposition \ref{existuhol}, parallel transport by $\nabla$ coincides with the unstable holonomy constructed in Proposition \ref{existuhol}, and thus $V$ is parallel with respect to $\nabla$ along unstable leaves. 

For a given unstable leaf $W^{u}(p)$ we can then find parallel vector fields $X_{1},\dots,X_{k}$ spanning the restriction of $V$ to $W^{u}(p)$. These vector fields are $C^{1}$ since $\nabla|W^{u}(p)$ is $C^{1}$. The restriction of $\nabla$ to $W^{u}(p)$ is torsion-free since $\eta$ vanishes on $W^{u}$. For $C^{1}$ vector fields, there is still a well-defined Lie bracket, and the Frobenius theorem characterizing integrability of a distribution remains true \cite{RS}. Since 
\[
0 = \nabla_{X_{i}}X_{j}- \nabla_{X_{j}}X_{i} = [X_{i},X_{j}], \; 1 \leq i,j\leq k
\] 
we then conclude via the $C^{1}$ Frobenius theorem that $V$ is a uniquely integrable subbundle of $TW^{u}$. Hence there is a $C^{2}$ foliation $\mathcal{V}$ of $SM$ which is tangent to $V$, such that each of the leaves $\mathcal{V}(p)$ is contained within the corresponding unstable leaf $W^{u}(p)$. 

Then $\mathcal{V}$ lifts to a foliation $\widetilde{\mathcal{V}}$ of $S\widetilde{M}$ which is invariant under the lifted action of $g^{t}$ and the action of $\pi_{1}(M)$. We adopt the notation of the Hopf parametrization described in Section \ref{amenred}. For each $x \in S\widetilde{M}$ there is a homeomorphism $\pi_{x}: \widetilde{W}^{u}(x) \rightarrow \p \widetilde{M} \backslash \{x_{-}\}$ given by projection, where $x_{-}$ is the negative endpoint of the geodesic through $x$ on $\p \widetilde{M}$. Then for a pair of points $x,y \in S\widetilde{M}$, we consider the homeomorphism 
\[
\pi_{y}^{-1} \circ \pi_{x}: \widetilde{W}^{u}(x) \backslash \{\pi_{x}^{-1}(y_{-})\} \rightarrow \widetilde{W}^{u}(y) \backslash \{\pi_{y}^{-1}(x_{-})\} 
\]
The homeomorphism $\pi_{y}^{-1} \circ \pi_{x}$ is easily described in terms of the global product structure of $S\widetilde{M}$:  for a point $z \in \widetilde{W}^{u}(x) \backslash \{\pi_{x}^{-1}(y_{-})\}$, $\pi_{y}^{-1} (\pi_{x}(z))$ is the unique intersection point of $\wt{W}^{cs}(z)$ and $\wt{W}^{u}(y)$. 

Since the Anosov splitting of $g^{t}$ is $C^{1}$, the map $\pi_{y}^{-1} \circ \pi_{x}$ is $C^{1}$ and the derivative is given by parallel transport with respect to the Kanai connection $\nabla$, which coincides with the global center stable holonomy map $h^{cs}$ for $Dg^{t}|E^{u}$ by the uniqueness statement in Proposition \ref{existuhol}. Since $\wt{V}$ is invariant under the action of $g^{t}$ and stable holonomy, $\wt{V}$ is invariant under center stable holonomy and therefore $D(\pi_{y}^{-1} \circ \pi_{x})(\wt{V}) = \wt{V}$. Since $\wt{V}$ is uniquely integrable, this implies for every $z \in \widetilde{W}^{u}(x)$ that $\pi_{y}^{-1}(\pi_{x}(\wt{\mathcal{V}}(z))) = \wt{\mathcal{V}}(\pi_{y}^{-1}(\pi_{x}(z))$. 

The homeomorphisms $\{\pi_{x}: \; x \in S\wt{M}\}$ form a system of charts for $\p \wt{M}$ which give $\p \wt{M}$ the structure of a $C^{1}$ manifold. The equivariance property of the foliation $\wt{\mathcal{V}}$ with respect to these charts implies that $\wt{\mathcal{V}}$ descends to a $C^{1}$ foliation $\mathcal{F}$ of $\p \wt{M}$. Furthermore, since $\wt{\mathcal{V}}$ is equivariant under the action of $\pi_{1}(M)$ (as it was lifted from a foliation $\mathcal{V}$ on $SM$), the foliation $\mathcal{F}$ is invariant under the action of $\pi_{1}(M)$ on $\p \wt{M}$. But every $\pi_{1}(M)$-invariant continuous foliation of $\p \wt{M}$ must be trivial, i.e., either for every $\xi \in \p \wt{M}$ we have $\mathcal{F}(\xi) = \{\xi\}$ or for every $\xi \in \p \wt{M}$ we have $\mathcal{F}(\xi) = \p \wt{M}$. This is proved in Section 4 of \cite{Ham1}; see also \cite{F}. This implies that $V = \{0\}$ or $V = E^{u}$, which completes the proof. 
\end{proof}

\begin{proof}[Proof of Theorem \ref{perhyp}] Since $g^{t}$ is a contact Anosov flow preserving the contact form $\theta$ with $\ker \theta = E^{u} \oplus E^{s}$ and $d\theta|\ker \theta$ being nondegenerate, the hypotheses of Theorem \ref{perhyp} imply that for any periodic point $p$, the eigenvalues of $Dg^{\l(p)}_{p}: E^{s}_{p} \rightarrow E^{s}_{p}$ are all equal in absolute value, and their common absolute value is the reciprocal of the absolute value of the eigenvalues of $Dg^{\l(p)}_{p}: E^{u}_{p} \rightarrow E^{u}_{p}$. As a consequence, for any periodic point $p$, $g^{t}$ is $\alpha$-bunched along the orbit of $p$ for any $\alpha < 2$.  The main result of Hasselblatt in \cite{Has1} then implies that $g^{t}$ is 1-bunched, so that the Anosov splitting of $g^{t}$ is $C^{1}$.  

Theorem \ref{perapprox} implies that for every ergodic $g^{t}$-invariant measure $\mu$, $\la_{+}(Dg^{t}|E^{u},\mu) = \la_{-}(Dg^{t}|E^{u},\mu)$. In particular, this holds when $\mu$ is the Liouville measure on $SM$, which as remarked earlier, is a fully supported ergodic invariant measure with local product measure for $g^{t}$. As remarked in Lemma \ref{hypirr}, $Dg^{t}|E^{u}$ is fiber bunched and so we can apply Lemma \ref{amen}: there is a finite cover $S\widetilde{M}$ of $SM$ for which the conclusions of Lemma \ref{amen} are satisfied. Since any lift of $g^{t}$ to a finite cover of $SM$ is itself the geodesic flow of a closed negatively curved manifold $\mathcal{M}$, we see that by Lemma \ref{hypirr}, the invariant flag constructed in Lemma \ref{amen} must be trivial, and thus by Lemma \ref{confcover} there must be a continuous holonomy invariant conformal structure on $E^{u}$ preserved by $Dg^{t}$. By Theorem 1 of \cite{Kan1}, if $Dg^{t}|E^{u}$ preserves a continuous conformal structure then $M$ is homotopy equivalent to a real hyperbolic manifold $N$ and there is a $C^{1}$ time-preserving conjugacy of the geodesic flow of $M$ to the geodesic flow of $N$. The minimal entropy rigidity theorem from \cite{BCG1} then implies that $M$ is homothetic to $N$. 
\end{proof}

\begin{proof}[Proof of Theorem \ref{hyp}] Hasselblatt \cite{Has2} proved that if the sectional curvatures of $M$ are relatively $1/4$-pinched, then $g^{t}$ is $1$-bunched and so the Anosov splitting of $g^{t}$ is $C^{1}$. The hypothesis that $\la_{+}(Dg^{t}|E^{u},\mu) = \la_{-}(Dg^{t}|E^{u},\mu)$ for the measure $\mu$ together with Lemma \ref{hypirr} then implies that $Dg^{t}$ preserves a conformal structure on $E^{u}$. The proof then concludes in the same manner as the proof of Theorem \ref{perhyp} above. 
\end{proof}

\begin{rem}\label{error}In this remark we explain the gap in \cite{Yue} mentioned in the introduction. First we recall the setting of the paper. The claim is that if $M$ is a closed $m$-dimensional negatively curved manifold and $Dg^{t}|E^{u}$ is measurably irreducible in the sense that there are no $Dg^{t}$-invariant measurable subbundles of $E^{u}$, then $Dg^{t}|E^{u}$ preserves a continuous conformal structure and therefore $M$ is homothetic to a real hyperbolic manifold by the same proof as given in Theorem \ref{perhyp}. In the first part of the remark we explain some flaws in the definition of boundedness for a conformal structure that is given in \cite{Yue}, and in the second part we explain how, even after correcting these flaws in the definition, the proof still appears to have a gap in proving boundedness at a critical step. 

As mentioned in Section \ref{contamen}, the conformal structures on $E^{u}$ can be topologized as a fiber bundle $\mathcal{C}$ over $SM$. Each fiber $\mathcal{C}_{x}$ may be identified with the nonpositively curved symmetric space $S(n):=SL(n,\R)/SO(n,\R)$, where $n = \dim E^{u} = m-1$. If we take this identification to be induced by a linear trivialization $E^{u}_{x} \rightarrow \R^{n}$, then it is unique up to an isometry of $S(n)$ and therefore $\mathcal{C}_{x}$ carries a canonical metric $\rho_{x}$ of nonpositive curvature. The bundle $\mathcal{C}$ over $SM$ has no distinguished section $SM \rightarrow \mathcal{C}$ and therefore in order to say that a conformal structure $\tau$  is ``bounded'' we thus have to compare it to a specific conformal structure $\tau_{0}: SM \rightarrow \mathcal{C}$ which we have chosen beforehand. This is handled properly in \cite{KS2}, in which a conformal structure $\tau$ is defined to be bounded if there is a \emph{continuous} conformal structure $\tau_{0}: SM \rightarrow \mathcal{C}$ and a constant $C > 0$ such that 
\[
\rho_{x}(\tau(x),\tau_{0}(x)) < C \; \text{for every $x \in SM$}.
\]
In \cite{Yue}, a measurable trivialization $E^{u} \rightarrow SM \times \R^{n}$ is fixed and a conformal structure on $E^{u}$ is then defined to be a measurable map $\tau: SM \rightarrow S(n)$. A measurable conformal structure is defined to be ``bounded'' if there is a constant $C > 0$ such that  $\rho(\tau(x),I) < C$, where $\rho$ is the nonpositively curved metric on $S(n)$ and $I$ is the image of the identity matrix in $S(n)$. In the definition of boundedness in \cite{KS2}, this corresponds to taking $\tau_{0}$ to be the section defined by pulling back the standard Euclidean metric on $\R^{n}$ via the measurable trivialization $E^{u} \rightarrow SM \rightarrow \R^{n}$. This is problematic because on page 747 of \cite{Yue} it is claimed that boundedness of the invariant conformal structure implies that $Dg^{t}|E^{u}$ is uniformly quasiconformal with respect to the \emph{continuous} conformal structure on $E^{u}$ defined by restricting the Riemannian metric on $T(SM)$ to $E^{u}$. But $\tau$ being a bounded distance from the measurable section $\tau_{0} \equiv I$ does not imply it is a bounded distance from any continuous conformal structure. 

Even after repairing this issue with the definition of boundedness, there is still an apparent gap in the argument which occurs on page 747 of \cite{Yue}. At this point measurable $g^{t}$-invariant affine connections $D^{s}$ and $D^{u}$ along the $W^{s}$ and $W^{u}$ foliations respectively have been constructed which are continuous when restricted to an individual $W^{s}$ and $W^{u}$ leaf respectively, but are only measurable in the transverse direction. We let $\mu$ denote the Liouville measure on $SM$. For $y \in W^{u}(x)$ we let $P_{xy}^{u}: E^{u}_{x} \rightarrow E^{u}_{y}$ be the parallel transport map with respect to $D^{u}$, and let $P_{xy}^{s}:E^{u}_{x} \rightarrow E^{u}_{z}$ be the analogous parallel transport map for $D^{s}$ with $z \in W^{s}(x)$ instead. A measurable conformal structure $\sigma^{u}$ has also been constructed on the unstable bundle $E^{u}$ which is $\mu$-a.e. parallel with respect to $D^{s}$ and $D^{u}$ in the following sense: for $\mu$-a.e. pair $x,y \in SM$ with $y \in W^{u}(x)$, there is a constant $\xi^{u}(x,y)$ such that for every $v,w \in E^{u}_{x}$
\[
\xi^{u}(x,y)\sigma^{u}(v,w) = \sigma^{s}(P^{s}_{xy}(v),P^{s}_{xy}(w))
\]
An analogous statement is true for parallel transport of $\sigma^{u}$ with respect to $P^{u}_{xy}$. It is claimed that this data implies that $\sigma^{u}$ is ``locally essentially bounded'' which as we've seen must be interpreted to mean that for each $p \in SM$, there is a neighborhood $U$ of $p$, a constant $C > 0$, and a continuous section $\tau_{0}: U \rightarrow \mathcal{C}|_{U}$ such that 
\[
\rho_{x}(\sigma^{u}(x),\tau_{0}(x)) < C \; \text{for $\mu$-a.e. $x \in U$}.
\]
The invariance of $\sigma^{u}$ under $D^{u}$ and $D^{s}$ together with the fact that $P^{u}$ and $P^{s}$ induce isometries between the fibers of $\mathcal{C}$ gives, for $y \in W^{u}(x)$, $z \in W^{s}(x)$,
\[
\rho_{y}(\sigma^{u}(y),(P^{u}_{yx})^{*}\tau_{0}(x)) = \rho_{x}(\sigma^{u}(x),\tau_{0}(x)) = \rho_{z}(\sigma^{u}(z),(P^{s}_{zx})^{*}\tau_{0}(x)).
\]
This does not allow us to compare $\rho_{x}(\sigma^{u}(x),\tau_{0}(x))$ to $\rho_{y}(\sigma^{u}(y),\tau_{0}(y))$ unless we also have uniform bounds on $\rho_{y}(\tau_{0}(y),(P^{u}_{yx})^{*}\tau_{0}(x))$. But the parallel transport maps $P^{u}_{xy}$ and $P^{s}_{xz}$ depend only measurably on $x,y,z$ and so, for instance, $\rho_{y}(\tau_{0}(y),(P^{u}_{yx})^{*}\tau_{0}(x))$ could grow arbitrarily large as $x,y$ vary through the neighborhood $U$ of $p$. In particular, there is no reason for $P^{u}$ and $P^{s}$ to behave nicely with respect to some continuous conformal structure on $E^{u}$ over $U$. This point is not addressed in \cite{Yue} and the proof appears incomplete as a result. 


\end{rem}

\end{section}

\begin{section}{Horizontal Subbundles}\label{horsub}
In this section we assume that $g^{t}: X \rightarrow X$ is an Anosov flow defined on a Riemannian manifold $X$ and that there exists a dominated splitting $E^{u} = H^{u} \oplus V^{u}$ of the unstable bundle for which $V^{u}$ is the most expanding bundle. We will refer to  $H^{u}$ as the \emph{horizontal} unstable bundle and $V^{u}$ as the \emph{vertical} unstable bundle. 

\begin{prop}\label{smoothvert}
$V^{u}$ is uniquely integrable with smooth leaves. The resulting foliation $W^{vu}$ is smooth when restricted to $W^{u}$ leaves. 
\end{prop}

\begin{proof}
Consider $f:=g^{1}$ as a partially hyperbolic map with invariant splitting $E^{u}_{f} \oplus E^{c}_{f} \oplus E^{s}_{f}$, where $E^{u}_{f} = V^{u}$, $E^{c}_{f} = H^{u} \oplus E^{c}$, and $E^{s}_{f} = E^{s}$. The statements of the proposition then follow from standard results in the theory of partially hyperbolic diffeomorphisms\cite{HPS}. 
\end{proof}

For each $p \in X$, we define an equivalence relation $\sim$ on points $x,y \in W^{u}(p)$ by $x \sim y$ if $x \in W^{vu}(y)$. We let $Q^{u}(p)$ be the quotient of $W^{u}(p)$ by this equivalence relation, which can be identified with the space of $W^{vu}$ leaves inside of $W^{u}(p)$, and we let $\Pi:W^{u} \rightarrow Q^{u}$ be the projection map. The next proposition verifies that the leaves of the $W^{vu}$ foliation are properly embedded in $W^{u}$, which implies that $Q^{u}(p)$ is a smooth manifold diffeomorphic to $\R^{k}$, $k = \dim H^{u}$. 

\begin{prop}\label{embed}
For each $p \in X$, there is a smooth embedding $\iota_{p}: \R^{k} \rightarrow W^{u}(p)$ with $\iota_{p}(0) = p$ such that $\iota_{p}(\R^{k})$ meets each $W^{vu}$ leaf inside of $W^{u}(p)$ in exactly one point. 
\end{prop}

\begin{proof}
Let $f = g^{1}$ and consider this as a partially hyperbolic map as in Proposition \ref{smoothvert}. The theory of partially hyperbolic diffeomorphisms then tells us that there is some $r > 0$ such that on any ball of radius $r$ in $SM$, the foliation tangent to $E^{u}_{f}$ is trivial \cite{HPS}. Furthermore, since there is a foliation tangent to $E^{u}_{f} \oplus E^{c}_{f}$, and the unstable foliation $W^{u}_{f}$ tangent to $E^{u}_{f}$ always smoothly subfoliates $E^{u}_{f} \oplus E^{c}_{f}$, we can choose this trivialization to be smooth along $W^{u}$ leaves. Choose a sequence of times $t_{n} \rightarrow \infty$ such that $g^{-t_{n}}(p) \rightarrow p$ in $X$. For each $n \in \N$, let $D_{n,r}$ be the disk of radius $r$ centered at $g^{-t_{n}}(p)$ in $W^{u}(g^{-t_{n}(p)})$. 

By shrinking $r$ if necessary, we can assume that $g^{-t}$ is a contracting map on $D_{n,r}$ for each $n$ in the induced Riemannian metric on $W^{u}$, which implies that $g^{t_{n}-t_{s}}(D_{n,r}) \subset D_{s,r}$ for $s > n$. For each $n$, choose a compact transversal submanifold $K_{n} \subset D_{n,r}$ to the $W^{vu}$ foliation which contains $g^{-t_{n}}(p)$ and is tangent to $H^{u}_{g^{-t_{n}}(p)}$ at $g^{-t_{n}}(p)$. $K_{n}$ meets each leaf of the induced foliation of $D_{n,r}$ by $W^{vu}$ in exactly one point. 

Consider the collection of $k$-dimensional submanifolds $g^{t_{n}}(K_{n})$ of $W^{u}(p)$. We make three claims. First we claim that if a $W^{vu}$ leaf intersects $g^{t_{n}}(K_{n})$, then it intersects $g^{t_{s}}(K_{s})$ for any $s > n$. Second, we claim that each $W^{vu}$ leaf meets each submanifold $g^{t_{n}}(K_{n})$ in at most one point. Lastly, we claim that for each $W^{vu}$ leaf in $W^{u}(p)$, there is an $n \in \N$ such that $g^{t_{n}}(K_{n})$ intersects this leaf. 

For the first claim, if $s > n$, then $g^{-t_{s}}(g^{t_{n}}(K_{n})) \subset D_{s,r}$ by construction. Since $K_{s}$ is a full transversal inside of $D_{s,r}$, any $W^{vu}$ leaf intersecting $g^{-t_{s}}(g^{t_{n}}(K_{n}))$ also intersects $K_{s}$. By $g^{t}$-invariance of the $W^{vu}$ foliation, any $W^{vu}$ leaf intersecting $g^{t_{n}}(K_{n})$ thus also intersects $g^{t_{s}}(K_{s})$.

For the second claim, suppose that $W^{vu}(q)$ intersects $g^{t_{n}}(K_{n})$ in the points $q$ and $q'$, for $q \neq q'$. $W^{u}(p)$ is exponentially contracted under $g^{-t}$, so for large enough $s$, there will be a curve contained entirely in $g^{-t_{s}}(W^{vu}(q)) \cap D_{s,r}$ which joins $g^{-t_{s}}(q)$ to $g^{-t_{s}}(q')$. On the other hand, since the splitting $E^{u} =V^{u} \oplus H^{u}$ is dominated, as $s \rightarrow \infty$, the tangent spaces to $g^{t_{n}-t_{s}}(K_{n})$ are uniformly asymptotic to the sequence of planes $H^{u}_{g^{t_{s}}(p)}$. Thus for large enough $s$, $g^{t_{n}-t_{s}}(K_{n})$ will be a small disk that is almost parallel to $H^{u}_{g^{t_{s}}(p)}$; in particular it will meet each leaf of $W^{vu} \cap D_{s,r}$ in at most one point. But this contradicts the existence of the segment joining $g^{-t_{s}}(q)$ to $g^{-t_{s}}(q')$ inside of $g^{-t_{s}}(W^{vu}(q)) \cap D_{s,r}$. 

For the last claim, recall that $W^{u}(p)$ is defined as the set of points in $X$ asymptotic to the orbit of $p$ under $g^{-t}$. Since $g^{-t_{n}}(p) \rightarrow p$, it follows that for any $q \in W^{u}(p)$, there is some $n > 0$ such that $g^{-t_{n}}(q) \in D_{n,r}$; the last claim follows. 

Having proven those three claims, we now construct the desired embedding inductively. Set $U_{1}:= g^{t_{1}}(K_{1})$. To construct $U_{n}$ from $U_{n-1}$, take the submanifold $g^{t_{n}}(K_{n})$ of $W^{u}(p)$ and use the smoothess of the $W^{vu}$ foliation of $W^{u}$ to map $g^{t_{n}}(K_{n})$ smoothly onto a submanifold of $W^{u}(p)$ which contains $q \in U_{n-1}$ for each $q$ such that $W^{vu}(q) \cap g^{t_{n}}(K_{n})$ is nonempty. By the first claim $U_{n} \subset U_{s}$ for $s \geq n$. By the second and third claim, the submanifold $U:= \bigcup_{n=1}^{\infty}U_{n}$ meets each $W^{vu}$ leaf in exactly one point. Properness of the embedding follows from the fact that the $W^{vu}$ foliation is locally trivial and that $U$ meets each $W^{vu}$ leaf in only one point.  
\end{proof}
Next we build a $C^{1}$ $g^{t}$-invariant connection $\nabla$ on the tangent bundle $TQ^{u}$ to $Q^{u}$ which will correspond to a $g^{t}$-invariant connection on the bundle $E^{u}/V^{u}$ over $SM$. $\nabla$ will play the same role in the proof of Lemma \ref{chypirr} below as the Kanai connection in the proof of Lemma \ref{hypirr}. 
\begin{lem}\label{connect}
Suppose that $Dg^{t}|H^{u}$ is fiber bunched. Then there is a $C^{1}$, flat, torsion-free $g^{t}$-invariant connection $\nabla$ on $Q^{u}$. For points $p$, $q \in W^{u}(p)$ and $w \in H^{u}_{p}$, 
\[
D\Pi_{q}^{-1}\circ P_{\Pi_{v}(p)\Pi_{v}(q)} \circ D\Pi_{p} = h^{u}_{pq}(w)
\]
where $P$ is parallel transport with respect to $\nabla$. 
\end{lem}

\begin{proof}
Since $V^{u}$ is a smooth subbundle of $E^{u}$ when restricted to $W^{u}$, the quotient bundle $E^{u}/V^{u}$ over $W^{u}$ is smooth. The projection $E^{u} \rightarrow E^{u}/V^{u}$ induces a bundle isomorphism $H^{u} \rightarrow 
E^{u}/V^{u}$ which is equivariant with respect to the action of $Dg^{t}$ on $H^{u}$ and the induced action of $Dg^{t}$ on $E^{u}/V^{u}$. We push forward the Riemannian metric on $H^{u}$ to a Riemannian metric on $E^{u}/V^{u}$, with respect to which the induced action of $Dg^{t}$ is fiber bunched. The isomorphism $H^{u} \rightarrow E^{u}/V^{u}$ also induces an unstable holonomy $\bar{h}^{u}$ for the action of $Dg^{t}$ on $E^{u}/V^{u}$. Since $E^{u}/V^{u}$ has a smooth structure along $W^{u}$ leaves with respect to which $Dg^{t}$ is smooth and the action of $Dg^{t}$ on $E^{u}/V^{u}$ is 1-fiber bunched, the unstable holonomy $\bar{h}^{u}$ is $C^{1}$ along $W^{u}$ leaves. 

By the uniqueness of this unstable holonomy, we have the following alternative construction of $\bar{h}^{u}$. Take two compact transversals $K_{1}$ and $K_{2}$ to the $W^{vu}$ foliation which meet the same collection of $W^{vu}$ leaves (or equivalently, they have the same projection to $Q^{u}(p)$). The projection $E^{u} \rightarrow E^{u}/V^{u}$ induces natural bundle isomorphisms $TK_{i} \rightarrow E^{u}/V^{u}$ over each of these transversals. Then the derivative of the chart transition map $(\Pi|_{K_{2}})^{-1} \circ \Pi|_{K_{1}}$ is the unstable holonomy $\bar{h}^{u}$ when we make the identifiactions $TK_{i} \cong E^{u}/V^{u}$

The projection $\Pi: W^{u}(p) \rightarrow Q^{u}(p)$ is smooth and hence induces a derivative map $D\Pi: E^{u} \rightarrow TQ^{u}$ with $V^{u} = \ker D\Pi$. Hence for each $x \in W^{u}(p)$ the induced map $\overline{D\Pi}: E^{u}_{x}/V^{u}_{x} \rightarrow TQ^{u}_{\Pi(x)}$ is an isomorphism. For $w$, $z \in Q^{u}(p)$ which are the image of $x$ and $y \in W^{u}(p)$ respectively, we define $P_{wz}: TQ^{u}_{w} \rightarrow TQ^{u}_{z}$ by $P_{wz} = \overline{D\Pi}_{y} \circ \bar{h}_{xy} \circ \overline{D\Pi}_{x}^{-1}$. We claim that $P_{wz}$ does not depend on the preimages $x$ and $y$ of $w$ and $z$ which were chosen. Suppose that $x'$ and $y'$ are two other points projecting to $w$ and $z$ respectively. Then 
\begin{align*}
\overline{D\Pi}_{y'} \circ \bar{h}_{x'y'} \circ \overline{D\Pi}_{x'}^{-1} &= \overline{D\Pi}_{y} \circ (\overline{D\Pi}_{y}^{-1} \circ \overline{D\Pi}_{y'}) \circ \bar{h}_{x'y'} \circ (\overline{D\Pi}_{x'}^{-1} \circ \overline{D\Pi}_{x})  \circ \overline{D\Pi}_{x}^{-1} \\
&= \overline{D\Pi}_{y} \circ \bar{h}_{y'y} \circ \bar{h}_{x'y'} \circ \bar{h}_{xx'} \circ \overline{D\Pi}_{x}^{-1} \\
&= \overline{D\Pi}_{y} \circ \bar{h}_{xy} \circ \overline{D\Pi}_{x}^{-1} 
\end{align*}
where we have used the observation that the derivatives of the transition maps for $\Pi$ are given by the unstable holonomy $\bar{h}^{u}$, and also the properties of the unstable holonomy $\bar{h}^{u}$ itself. 

It's straightforward to check that $P_{wz}$ is equivariant with respect to the induced derivative action $\overline{Dg^{t}}: TQ^{u}(p) \rightarrow TQ^{u}(g^{t}(p))$, using the equivariance property of $\bar{h}^{u}$. $P_{wz}$ is also $C^{1}$ in the variables $w$ and $z$ and has the property that for $x,y,z \in Q^{u}(p)$, $P_{yz} \circ P_{xy} = P_{xz}$. This implies that for each $X \in TQ^{u}(p)$, 
\[
\mathcal{P}(X) = \{Y \in TQ^{u}: \; P_{xy}(X) = Y \; \text{for some $x,y \in Q^{u}(p)$}\}
\]
is a $C^{1}$ submanifold of $TQ^{u}$ which is transverse to the tangent spaces $TQ^{u}_{x}$. The tangent spaces to the foliation of $TQ^{u}$ by these subfoliations define an Ehresmann connection on $Q^{u}(p)$ which we can then use to define a connection $\nabla$ on $Q^{u}(p)$. The parallel transport of a vector by $\nabla$ is given by the linear maps $P_{wz}$. Thus $\nabla$ is a $C^{1}$ flat affine connection on $Q^{u}(p)$. Since the maps $P_{wz}$ are equivariant with respect to $\overline{Dg^{t}}$, $\nabla$ is also $g^{t}$ invariant. The definition of $P$ immediately implies the equation stated in the lemma. 

It only remains to show that $\nabla$ is torsion-free. Let $T$ be the torsion tensor of $\nabla$. $T$ is a mixed tensor of type $(2,1)$ on $TQ^{u}$ which is invariant under $g^{t}$. But the fact that $Dg^{t}|H^{u}$ is fiber bunched implies that $Dg^{t}$ acts by exponential contraction on tensors of type $(2,1)$ on $TQ^{u}$. This forces $T \equiv 0$ so that $\nabla$ is torsion-free. 
\end{proof}

The following lemma is fundamental to everything that follows in this paper. Recall that in the proof of Theorem \ref{hypirr}, one of the critical steps was to establish that the stable holonomy of the cocycle $Dg^{t}|E^{u}$ could be represented as the derivative of the holonomy map between unstable leaves induced by the center stable foliation. Lemma \ref{wcs} establishes the analogous property in our situation. 

Let $r > 0$ be small enough that all of the foliations $W^{*}$ under consideration are trivial on a ball of radius $r$. Given two points $x, y \in X$ with $y \in W^{cs}_{r}(x)$, there is then a well-defined $W^{cs}$-holonomy map $L_{xy}: W^{u}_{r}(x) \rightarrow W^{u}_{r}(y)$. For $z \in W^{u}_{r}(x)$, $L_{xy}(z)$ is defined to be the unique point in the intersection $W^{cs}_{r}(z) \cap W^{u}_{r}(y)$. In general this holonomy map is only H\"older continuous \cite{PSW}. We establish under proper fiber bunching assumptions on $Dg^{t}|H^{u}$ that $L_{xy}$ is differentiable when restricted to curves tangent to $H^{u}$. If we think of $g^{1}$ as a partially hyperbolic diffeomorphism as in Proposition \ref{smoothvert} with center bundle $H^{u} \oplus E^{c}$, Lemma \ref{wcs} can be viewed as an extension of Theorem B in \cite{PSW} to the case in which there may not be a foliation tangent to the center distribution. 

\begin{lem}\label{wcs}
Suppose that $E^{u}$ and $H^{u}$ are $\beta$-H\"older continuous and that $Dg^{t}|H^{u}$ is $\beta$-fiber bunched. Then the $W^{cs}$ holonomy map $L_{xy} : W^{u}_{r}(x) \rightarrow W^{u}_{r}(y)$ maps $C^{1}$ curves tangent to $H^{u}$ to $C^{1}$ curves tangent to $H^{u}$ and therefore $L_{xy}$ is differentiable along $H^{u}$. For $z \in W^{u}_{r}(x)$, the derivative of $L_{xy}$ along $H^{u}$ is given by 
\[
D_{z}L_{xy}|H^{u} = h^{cs}_{z L_{xy}(z)}
\]
\end{lem}
\begin{proof}
Let $x$, $y$ be two points in $X$ such that $x \in W_{r}^{cs}(y)$. Set $x_{n} = g^{n}x$ and $y_{n} = g^{n}y$. For each $n \geq 0$, choose a  hypersurface $S_{n}$ of uniform size and biLipschitz to an open subset of $\R^{2m}$ with Lipschitz constants independent of $n$ that is transverse to the direction of the flow $E^{c}$, and contains $W^{u}_{r}(x_{n})$ and $W^{u}_{r}(y_{n})$. Let $f_{n}: S_{n-1} \rightarrow S_{n}$ be the smooth map defined by $f_{n}(r) = g^{t(r)}(r)$, where $t(r)$ is the unique time, smoothly depending on $r$, with $t(x_{n-1}) = 1$ and such that $g^{t(r)} \in S_{n}$. $f_{n}$ is defined on a neighborhood of $x_{n-1}$ of uniform size, independent of $n$. Further, $f_{n}$ is uniformly hyperbolic on the interior of this neighborhood with the same contraction and expansion estimates (up to multiplicative constants) as $g^{1}$ on the stable and unstable bundles $E^{u}$ and $E^{s}$. Set $F^{n} = f_{n} \circ f_{n-1} \circ \dots \circ f_{1}$. Note that $F^{n}$ is defined on increasingly small neighborhoods of $x$ as $n \rightarrow \infty$; the only points for which $F^{n}$ is defined for all $n \geq 1$ are the points on the intersection of $W_{r}^{cs}(x)$ with $S := S_{0}$. 

Let $\beta$ be the minimum of the H\"older exponents of $H^{u}$ and $E^{u}$ viewed as subbundles of $TX$. As remarked in the Introduction, there is a $\beta$-H\"older system of linear identifications $I_{pq}:E^{u}_{p} \rightarrow E^{u}_{q}$ defined for $p$ near $q$ with $I_{pp}$ being the identity on $E^{u}_{p}$. We can choose these identifications so that $I_{pq}(H^{u}_{p}) = H^{u}_{q}$. For each $n$, let $A_{n}: W^{u}_{r}(x_{n}) \rightarrow W^{u}_{r}(y_{n})$ be a diffeomorphism with $A_{n}(x_{n}) = y_{n}$. Since the unstable foliation is H\"older continuous in the $C^{1}$ topology with H\"older exponent $\beta$, we can choose $A_{n}$ such that 
\[
\|I_{qp} \circ DA_{n}-Id\| \leq Cd(p,q)^{\beta} 
\]
\[
\|DA_{n} \circ I_{pq}-Id\| \leq Cd(p,q)^{\beta}
\]
for some constant $C > 0$ and $p \in W^{u}_{r}(x_{n})$, $q \in W^{u}_{r}(y_{n})$. For $z \in S_{n}$, let $\widehat{W}^{s}(z)$ denote the smooth projection of $W^{s}_{r}(z)$ onto $S_{n}$ along the orbit foliation $E^{c}$, given by using $g^{t}$ to flow these leaves onto $S_{n}$. Let $\widehat{H}^{u}$, $\widehat{E}^{u}$, and $\widehat{E}^{s}$ denote the projection of these subbundles onto $TS_{n}$ by flowing along the orbit foliation. 

Let $\varphi$ be the holonomy map between $W^{u}_{r}(x)$ and $W^{u}_{r}(y)$ induced by the projected stable foliation $\widehat{W}^{s}$. Let $\varphi_{n} = F^{-n} \circ A_{n} \circ F^{n}$, which is defined on a neighborhood of $x$ (dependent on $n$) inside of $W^{u}_{r}(x)$. Let $\gamma:[-1,1]$ be a $C^{1}$ curve tangent to $H^{u}$ inside of $W^{u}_{r}(x)$ with $\gamma(0) = x$. Our first goal is to prove that $\varphi \circ \gamma$ is differentiable at 0, i.e., that the image of the curve $\gamma$ under $\widehat{W}^{s}$-holonomy along the transversal $S$ is differentiable at $p$. 

We first claim that the sequence of linear maps $\left\{(DF^{-n}_{y}\circ DA_{n} \circ DF^{n}_{x})|H^{u}_{x}: \, n \in \N\right\}$ is Cauchy (note that we have restricted the domain of these maps to $H^{u}_{x}$). We closely follow the proof of Proposition \ref{existuhol} given in \cite{KS}. We begin with the formula 
\begin{align*}
(DF^{n}_{y})^{-1} \circ  DA_{n} \circ DF^{n}_{x} = DA_{0} + \sum_{i=0}^{n-1}(DF^{i}_{y})^{-1} \circ R_{i} \circ DF^{i}_{x}
\end{align*}
where $R_{i} = (D_{y_{i}}f_{i+1})^{-1} \circ DA_{i+1} \circ D_{x_{i}}f_{i+1}-DA_{i}$. For the rest of the proof we will consider all linear maps as restricted to $\widehat{E}^{u}$ for the purpose of calculating norms. We want to estimate the product 
\begin{align*}
\|(DF^{n}_{y})^{-1}\| \cdot \|DF^{n}_{x}|\widehat{H}^{u}\| &\leq \prod_{i=0}^{n-1}\|(D_{y_{i}}f_{i})^{-1}\| \cdot \|D_{x_{i}}f_{i}|\widehat{H}^{u}\| \\
&= \left(\prod_{i=0}^{n-1}\|(D_{y_{i}}f_{i})^{-1}\| \cdot \|(D_{x_{i}}f_{i})^{-1}|\widehat{H}^{u}\|^{-1}\right) \\ &\cdot \left(\prod_{i=0}^{n-1}\|(D_{x_{i}}f_{i})^{-1}|\widehat{H}^{u}\| \cdot \|D_{x_{i}}f_{i}|\widehat{H}^{u}\|\right)
\end{align*}
To bound the first factor, we observe that $\|(D_{x_{i}}f_{i})^{-1}|\widehat{H}^{u}\| = \|(D_{x_{i}}f_{i})^{-1}\| $ since $\widehat{H}^{u}$ is the less expanded term of the dominated splitting $\widehat{E}^{u} = \widehat{V}^{u} \oplus \widehat{H}^{u}$. We then use the estimate
\begin{align*}
\frac{\|(D_{y_{i}}f_{i})^{-1}\|}{\|(D_{x_{i}}f_{i})^{-1}\|} &\leq \frac{\|(D_{y_{i}}f_{i})^{-1}-I_{x_{i}y_{i}} \circ (D_{x_{i}}f_{i})^{-1} \circ I_{x_{i+1}y_{i+1}}^{-1}\|}{\|(D_{x_{i}}f_{i})^{-1}\|}\\
&+\frac{\|I_{x_{i}y_{i}} \circ (D_{x_{i}}f_{i})^{-1} \circ I_{x_{i+1}y_{i+1}}^{-1}\|}{\|(D_{x_{i}}f_{i})^{-1}\|} \\
&\leq C'd(x_{i},y_{i})^{\beta} + 1
\end{align*}
for some constant $C'$. Here we use the fact that $\|I_{pq}\|$ is uniformly bounded when $p$ and $q$ are close (say $d(p,q) \leq r$), and that the derivative $D_{p}f_{i}$ is smooth as a function of $p$, hence when we use the identifications $I_{pq}$, it becomes H\"older with H\"older exponent $\beta$. 

To bound the second factor, we note that $Df_{i}|\widehat{H}^{u}$ is fiber bunched since the cocycle $Dg^{t}|H^{u}$ we derived it from was fiber bunched. Hence there is a constant $\delta < 1$ such that 
\[
\|(D_{p}f_{i})^{-1}|\widehat{H}^{u}\| \cdot \|D_{p}f_{i}|\widehat{H}^{u}\| \leq \|D_{p}f_{i}|\widehat{E}^{s}\|^{-\beta} \delta 
\]
for all $p \in S_{i}$, where $\delta$ is independent of $i$. 

Putting these two bounds together, we obtain
\[
\|(DF^{n}_{y})^{-1}\| \cdot \|DF^{n}_{x}|\widehat{H}^{u}\| \leq \prod_{i=0}^{n-1}(C'd(x_{i},y_{i})^{\beta}+1)  \prod_{i=0}^{n-1} \delta \|D_{x_{i}}f_{i}|\widehat{E}^{s}\|^{-\beta} 
\]
The first product is uniformly bounded since $d(x_{i},y_{i}) \rightarrow 0$ exponentially in $i$, so we get a constant $C''$ such that 
\[
\|(DF^{n}_{y})^{-1}\| \cdot \|DF^{n}_{x}|\widehat{H}^{u}\| \leq C''\delta^{n}\prod_{i=0}^{n-1}\|D_{x_{i}}f_{i}|\widehat{E}^{s}\|^{-\beta}
\]
Now we can also estimate 
\begin{align*}
\|R_{i}\| &\leq \|(D_{y_{i}}f_{i+1})^{-1} \circ DA_{i+1}\| \cdot \|D_{x_{i}}f_{i+1}-DA_{i+1}^{-1} \circ D_{y_{i}}f_{i+1} \circ DA_{i}\| \\
&\leq Cd(x_{i},y_{i})^{\beta} \\
&\leq Cd(x,y)^{\beta}\prod_{i=0}^{n-1}\|D_{x_{i}}f_{i}|\widehat{E}^{s}\|^{\beta}
\end{align*}
for some constant $C$. In the first inequality we used the H\"older closeness of $DA_{i}$ to the identity, together with uniform bounds on the norms of all of the linear maps involved. In the second inequality we use the fact that $x$ and $y$ lie on the same stable manifold in $S$. We have the basic bound
\begin{align*}
\|(DF^{n}_{y})^{-1} \circ DA_{n} \circ DF^{n}_{x} - DA_{0}|\widehat{H}^{u}\| &\leq \sum_{i=0}^{n-1}\|(DF^{i}_{y})^{-1} \circ R_{i} \circ DF^{i}_{x}|\widehat{H}^{u}\| \\
&\leq \sum_{i=0}^{n-1}\|(DF^{i}_{y})^{-1}\| \cdot \|DF^{i}_{x}|\widehat{H}^{u}\| \cdot \|R_{i}\| 
\end{align*}
We replace the right side with the previously obtained bounds on the factors $\|(DF^{i}_{y})^{-1}\| \cdot \|DF^{i}_{x}|\widehat{H}^{u}\|$ and $\|R_{i}\|$. This gives an upper bound of 
\begin{align*}
\sum_{i=0}^{n-1} \left( C''\delta^{i}\prod_{j=0}^{i-1}\|D_{x_{i}}f_{i}|\widehat{E}^{s}\|^{-\beta} \cdot Cd(x,y)^{\beta}\prod_{j=0}^{i-1}\|D_{x_{i}}f_{i}|\widehat{E}^{s}\|^{\beta} \right) \leq C^{*}d(x,y)^{\beta}
\end{align*}
for some constant $C^{*}$. Also note that 
\begin{align*}
\|(DF^{n+1}_{y})^{-1} \circ DA_{n+1} \circ DF^{n+1}_{x} &- (DF^{n}_{y})^{-1} \circ DA_{n} \circ DF^{n}_{x}|\widehat{H}^{u}\| \\
&= \|(DF^{n}_{y})^{-1} \circ R_{n} \circ DF^{n}_{x}|\widehat{H}^{u}\| \\
&\leq C^{*}\delta^{n}d(x,y)^{\beta}
\end{align*}
This second inequality immediately implies that the sequence of linear maps 
\[
\left\{(DF^{n}_{y})^{-1}\circ DA_{n} \circ DF^{n}_{x}|H^{u}_{x}: \, n \in \N\right\}
\]
is Cauchy. Hence this sequence converges to a linear map $T_{xy}: \widehat{H}^{u}_{x} \rightarrow \widehat{E}^{u}_{y}$. However, for any given vector $v \in \widehat{H}^{u}_{x}$, $DA_{n} \circ DF^{n}_{x}(v)$ is a vector which makes an angle $\theta_{n}$ with $\widehat{H}^{u}_{y}$, where $\theta_{n}$ is uniformly bounded away from $\pi/2$, independent of $n$. Applying $DF^{-n}_{y}$ exponentially contracts this angle since the splitting $\widehat{E}^{u} = \widehat{V}^{u} \oplus \widehat{H}^{u}$ is dominated, so letting $n \rightarrow \infty$, we conclude that $T_{xy}$ must have image in $\widehat{H}^{u}_{y}$. 

For each $j \geq 0$, we can also consider the sequence of linear maps
\[
\left\{(DF^{n+j}_{y} \circ (DF^{j}_{y})^{-1})^{-1}\circ DA_{n+j} \circ DF^{n+j}_{x} \circ (DF^{j}_{x})^{-1}|H^{u}_{x}: \, n \in \N\right\}
\]
For the same reasons as for the original sequence, this sequence is Cauchy and converges to a limit that we denote $T_{x_{j}y_{j}}$ which is a linear map from $\widehat{H}^{u}_{x_{j}}$ to $\widehat{H}^{u}_{y_{j}}$. It is straightforward to check that for each $n$ we have $(DF^{n}_{y})^{-1} \circ T_{x_{n}y_{n}} \circ DF^{n}_{x} = T_{xy}$ by writing out the limiting expression for $T_{x_{n}y_{n}}$. Since we chose the transversal $S$ to contain $W^{u}_{r}(x)$ and $W^{u}_{r}(y)$, we have $\widehat{H}^{u}_{x} = H^{u}_{x}$ and the same for $y$. We now consider the center stable holonomy map $h^{cs}_{xy}:H^{u}_{x} \rightarrow H^{u}_{y}$. This is equivariant with respect to $DF^{n}$ as well and also depends in a $\beta$-H\"older manner on the points $x$ and $y$.  Then 
\begin{align*}
\|h^{cs}_{xy}-P_{xy}\| &= \|(DF^{n}_{y})^{-1} \circ (h^{cs}_{x_{n}y_{n}} - T_{x_{n}y_{n}}) \circ DF^{n}_{x}|H^{u}_{x}\| \\
&\leq \|(DF^{n}_{y})^{-1}\| \cdot \| DF^{n}_{x}|H^{u}_{x}\|\cdot \|h^{cs}_{x_{n}y_{n}} - T_{x_{n}y_{n}}\| \\
&\leq C\delta^{n}\prod_{i=0}^{n-1}\|D_{x_{i}}f_{i}|\widehat{E}^{s}\|^{-\beta}d(x_{n},y_{n})^{\beta} \\
&\leq C^{*}\delta^{n}
\end{align*}
for some constant $C^{*}$. As $n \rightarrow \infty$, $\delta^{n} \rightarrow 0$, so $h^{cs}_{xy}=T_{xy}$. 

To prove differentiability of $\varphi \circ \gamma$, take a coordinate chart on $S$ (as well as each of the transversals $S_{n}$) so that we can work with the linear structure on $\R^{2m}$. Let $y$ correspond to the origin. We will not change the notation for the maps, so they should be understood in this chart. Let $v = \gamma'(0)$. We need to show that $\varphi(\gamma(s))$ agrees with its claimed linearization $s \cdot h^{cs}_{xy}(v)$ to first order at the origin. First observe that the calculations above are valid if we replace $x$ and $y$ by any two points $x'$, $y'$ in $S$ such that $y' \in \widehat{W}^{s}_{loc}(x')$, whenever $n$ is small enough (relative to $x'$ and $y'$) that the iterates $F,F^{2},\dots,F^{n}$ are all defined on a neighborhood of $x'$ and $y'$. This implies that 
\[
\|(DF^{n})^{-1} \circ DA_{n} \circ DF^{n}_{\gamma(s)}(\gamma'(s)) - DA_{0}(\gamma'(s))\| \leq C|s|^{\beta}
\]
whenever $s$ is small enough that $F^{n}$ is defined on a neighborhood of $\gamma(s)$ and $A_{n}(F^{n}(\gamma(s)))$ lies in the image of $F^{n}$. The constant $C$ is independent of $n$, so $(DF^{n})^{-1} \circ DA_{n} \circ DF^{n}_{\gamma(s)}(\gamma'(s))$ is a H\"older continuous function of $s$ with H\"older exponent and constant independent of $n$ for $|s|$ small. Note that $A_{n}(F^{n}(\gamma(s)))$ will not necessarily lie on $\widehat{W}^{s}(F^{n}(\gamma(s)))$, but it will be $\beta$-H\"older close to the intersection of $\widehat{W}^{s}(F^{n}(\gamma(s)))$ with $W^{u}_{r}(y_{n})$, so our estimates remain valid.  By the mean value inequality, we thus obtain 
\[
\|\varphi_{n}(\gamma(s)) - s \cdot D\varphi_{n}(\gamma'(0))\| \leq C|s|^{1+\beta}
\]
for a constant $C$. 

We next estimate the difference between $\varphi$ and $\varphi_{n}$ near $\gamma(0)$. Observe that $\varphi = (F^{n})^{-1} \circ \psi_{n} \circ F^{n}$, where $\psi_{n}$ is the $\widehat{W}^{s}$-holonomy map from $\widehat{W}_{r}^{u}(x_{n})$ to $W^{u}_{r}(y_{n})$. Hence for $s$ small enough that $\gamma(s)$ is in the domain of definition of the expressions below, 
\begin{align*}
\|\varphi_{n}(\gamma(s))-\varphi(\gamma(s))\| &= \|((F^{n})^{-1} \circ A_{n} \circ F^{n} - (F^{n})^{-1} \circ \psi_{n} \circ F^{n})(\gamma(s))\| \\
&\leq C \|(DF^{n})^{-1}|E^{u}\| \cdot \|A_{n} \circ F^{n}(\gamma(s)) - \psi_{n} \circ F^{n}(\gamma(s))\| 
\end{align*}
since $F^{-n}$ exponentially contracts distances on unstable leaves. Next we note that $\psi_{n}$ and $A_{n}$ are $\beta$-H\"older close in the $C^{0}$ topology. As a consequence, since they both map $x$ to $y$,
\begin{align*}
C \|(DF^{n})^{-1}|E^{u}\| &\cdot \|A_{n} \circ F^{n}(\gamma(s)) - \psi_{n} \circ F^{n}(\gamma(s))\| \\
&\leq C \|(DF^{n})^{-1}|E^{u}\| d(F^{n}(x),F^{n}(y))^{\beta}\|F^{n}(\gamma(s))\| \\
&\leq C\|(DF^{n})^{-1}|E^{u}\|\cdot \|DF^{n}|E^{s}\|^{\beta}\cdot \|F^{n}(\gamma(s))\| \\
&\leq C\|(DF^{n})^{-1}|E^{u}\|\cdot \|DF^{n}|E^{s}\|^{\beta}\cdot \|DF^{n}|\widehat{H}^{u}\| \cdot |s| \\
&\leq C\delta^{n}|s|
\end{align*}
where we have not paid much attention to the constant $C$ in front (which will change from line to line). In the third line we use the exponential contraction of stable leaves by $F^{n}$, and in the fourth line we use the fiber bunching property on $H^{u}$ transferred to the induced bundle $\widehat{H}^{u}$, noting that $\|(DF^{n})^{-1}|E^{u}\| = \|(DF^{n})^{-1}|\widehat{H}^{u}\|$. 

We now compare $\varphi \circ \gamma$ to the linearization $h^{cs}_{xy}(v) \cdot s$ at 0. Fix $n \in \N$. For $|s|$ small enough that all of the expressions above are defined for this $n$, we obtain 
\begin{align*}
\|\varphi (\gamma(s)) - h^{cs}_{xy}(v) \cdot s\| &\leq \|\varphi(\gamma(s))-\varphi_{n}(\gamma(s))\| + \|\varphi_{n}(\gamma(s)) - s \cdot D\varphi_{n}(v)\| \\
&+ |s| \cdot \|D\varphi_{n}(v)-h^{cs}_{xy}(v)\| \\
&\leq C(\delta^{n}|s| + |s|^{1+\beta} + |s| \cdot \|D\varphi_{n}(\gamma'(0))-h^{cs}_{xy}(v)\|)
\end{align*}
Dividing through by $|s|$, we obtain 
\[
\frac{\|\varphi (\gamma(s)) - h^{cs}_{xy}(v) \cdot s\|}{|s|} \leq C(\delta^{n} + |s|^{\beta} + \|D\varphi_{n}(v)-h^{cs}_{xy}(v)\|)
\]
We can consider $n:=n(s)$ as an integer function of $s$ such that $n(s) \rightarrow \infty$ as $s \rightarrow 0$. Then as $s \rightarrow 0$, the right side converges to 0. We thus obtain that $\varphi \circ \gamma$ agrees to first order with its linearization at 0, i.e., $\varphi \circ \gamma$ is differentiable at 0, and furthermore, $(\varphi \circ \gamma)'(0) = h^{cs}_{\gamma(0)\varphi(\gamma(0))}(\gamma'(0))$. 

Now observe that holonomy from $W^{u}_{r}(x)$ to $W^{u}_{r}(y)$ along the projected stable foliation $\widehat{W}^{s}$ corresponds precisely to $W^{cs}$-holonomy in $X$. Hence the curve $\varphi \circ \gamma$ is also the image of $\gamma$ under the $W^{cs}$-holonomy $L_{xy}$. We can apply our calculations to the other points of $\gamma$ by recentering at each pair of points $x'$, $y'$ lying on $\gamma$ and $\varphi \circ \gamma$ respectively with $y' \in W^{cs}_{r}(x')$. This proves that $\varphi \circ \gamma$ is differentiable for every $t \in [-1,1]$, and furthermore we have the derivative formula
\[
(\varphi \circ \gamma)'(t) = h^{cs}_{\gamma(t)\varphi(\gamma(t))}(\gamma'(t))
\]
which completes the proof. 
\end{proof}
\end{section}

\begin{section}{Proof of Theorem 1.3}\label{thirdthm}
In this section we prove Theorem \ref{weakchyp} using the results of Section \ref{horsub}. We first give a description of the neighborhood $\mathcal{U}$ of the symmetric metric $d_{\mathbb{K}}$ referred to in the hypotheses of Theorem \ref{weakchyp} and \ref{regchyp}. As described in the Introduction there is a $C^{2}$ open neighborhood $\mathcal{U}^{*}$ of $d_{\mathbb{K}}$ such that for any $d \in \mathcal{U}^{*}$ the geodesic flow $g^{t}$ of $d$ admits dominated splittings $E^{u} = H^{u} \oplus V^{u}$ and $E^{s} = H^{s} \oplus V^{s}$ of the stable and unstable bundles. For each metric $d$ there is a constant $C \geq 1$ and $a < b < c < \l$ such that for every $t \geq 0$,
\[
C^{-1}e^{-bt} \leq \|Dg^{-t}|H^{u}\| \leq Ce^{-at},
\]
\[
C^{-1}e^{-\l t} \leq \|Dg^{-t}|V^{u}\| \leq Ce^{-ct},
\]
with $C \rightarrow 1$, $a,b \rightarrow 1$, and $c,\l \rightarrow 2$ as $d \rightarrow d_{\mathbb{K}}$ in the $C^{2}$ topology on Riemannian metrics. Using the flip map $p \rightarrow -p$ which is an isometry for the Sasaki metric on $SM$ we conclude that the above inequalities also hold with $H^{s}$ and $V^{s}$ replacing $H^{u}$ and $V^{u}$ and taking $t \geq 0$ instead. 

We want to show that for $a,b$ close enough to $1$ and $c,\l$ close enough to $2$ there is a $\beta > 0$ such that both $E^{u}$ and $H^{u}$ are $\beta$-H\"older continuous subbundles of $T(SM)$ and $Dg^{t}|H^{u}$ is $\beta$-fiber bunched. This then implies that we can apply the results of Section \ref{horsub} to the geodesic flow of metrics $d$ in a small enough $C^{2}$-neighborhood $\mathcal{U} \subset \mathcal{U}^{*}$ of $d_{\mathbb{K}}$. 

First we compute the H\"older regularity of $E^{u}$ using results of Hasselblatt \cite{Has2} on the regularity of the stable and unstable bundles for Anosov flows. Recall from the beginning of Section \ref{firstthm} that for $0 < \alpha < 2$ we say that $g^{t}$ is $\alpha$-bunched if there is some $T > 0$ such that for $t \geq T$
\[
\|Dg^{t}_{p}|E^{u}_{p}\|^{\alpha} \cdot \|Dg^{t}_{p}|E^{s}_{p}\| \cdot \|(Dg^{t}_{p})^{-1}|E^{u}_{g^{t}(p)}\| < 1, \; t \geq T, p \in SM
\]
The results of \cite{Has2} imply that if $g^{t}$ is $\alpha$-bunched then $E^{u}$ is $C^{\alpha-\e}$ for any $\e > 0$. Plugging in the bounds on $Dg^{t}|E^{u}$ and $Dg^{t}|E^{s}$ described above, we see that for $t \geq 0$, 
\[
\|Dg^{t}_{p}|E^{u}_{p}\|^{\alpha} \cdot \|Dg^{t}_{p}|E^{s}_{p}\| \cdot \|(Dg^{t}_{p})^{-1}|E^{u}_{g^{t}(p)}\| \leq C^{3}e^{(\l\alpha - 2a)t}
\]
Thus $g^{t}$ is $\alpha$-bunched if and only if $\alpha < \frac{2a}{\l}$. As $a \rightarrow 1$ and $\l \rightarrow 2$, we see that $\frac{2a}{\l} \rightarrow 1$, so given any $0 < \alpha <1$ we can always take the neighborhood $\mathcal{U}$ small enough that $E^{u}$ is $C^{\alpha}$ for $d \in \mathcal{U}$.  

For the regularity of $H^{u}$ we refer to the $C^{r}$ section theorem in \cite{Shu}. From this theorem we deduce that $H^{u}$ is $\beta$-H\"older continuous for any $\beta < \frac{b}{c}$. The ratio $\frac{b}{c}$ converges to $\frac{1}{2}$ as $b \rightarrow 1$, $c \rightarrow 2$, so by taking $d$ close enough to $d_{\mathbb{K}}$ we can assume $H^{u}$ is $\beta$-H\"older for any given $\beta < \frac{1}{2}$. It is then clear that if $a,b$ are close enough to 1 then $Dg^{t}|H^{u}$ is $\beta$-fiber bunched for $\beta$ close to $\frac{1}{2}$. 

We take $d$ to lie in the neighborhood $\mathcal{U}$ of $d_{\mathbb{K}}$ described above. We assume that there is a $g^{t}$-invariant ergodic, fully supported measure $\mu$ with local product structure such that $\la_{+}(Dg^{t}|H^{u},\mu) = \la_{-}(Dg^{t}|H^{u},\mu)$. We will not need Assumption (2) of \ref{weakchyp} until the proof of Lemma \ref{reg}. 

The first lemma we prove is the analogue of Lemma \ref{hypirr} for the horizontal bundle $H^{u}$.  For a continuous subbundle $\E$ of $T(SM)$ and a point $p \in SM$, we define the $\E$-accessibility class $\mathscr{A}(p;\E)$ of $p$ to be the set of all points $q \in SM$ which can be joined to $p$ by a piecewise $C^{1}$ curve $\gamma$ tangent to $\E$.

\begin{lem}\label{chypirr}
Let $\E \subset H^{u}$ be a nonzero measurable $g^{t}$-invariant subbundle. Then $\E = H^{u}$. 
\end{lem}

\begin{proof}
By Lemma \ref{contsub}, $\E$ coincides $\mu$-a.e. with a continuous $g^{t}$-invariant, holonomy invariant subbundle of $H^{u}$, which we will also denote by $\E$. We claim that $\mathscr{A}(p;\E)$ is dense in $W^{u}(p)$ for every $p \in SM$. Since $\E$ is $g^{t}$-invariant, $\mathscr{A}(g^{t}p;\E) = g^{t}(\mathscr{A}(p;\E))$ for every $t \in \R$. Pass to the universal cover $S\widetilde{M}$ and note that for every $\gamma \in \Gamma := \pi_{1}(M)$, we also have $\mathscr{A}(D\gamma(p);\E) = D\gamma(\mathscr{A}(p;\E))$ for $p \in S\wt{M}$, since the lifted bundle $\tilde{\E}$ is invariant under $\Gamma$. As in the proof of Lemma \ref{hypirr}, we let $\pi_{x}: \wt{W}^{u}(x) \rightarrow \p \wt{M} \backslash \{x_{-}\}$ be the projection homeomorphism onto the boundary. For $x, y \in S \wt{M}$, consider as before the transition homeomorphism 
\[
\pi_{y}^{-1} \circ \pi_{x}: \widetilde{W}^{u}(x) \backslash \{\pi_{x}^{-1}(y_{-})\} \rightarrow \widetilde{W}^{u}(y) \backslash \{\pi_{y}^{-1}(x_{-})\} 
\]
Lemma \ref{wcs} implies that $\pi_{y}^{-1} \circ \pi_{x}$ is differentiable when restricted to $C^{1}$ curves tangent to $H^{u}$, and that the derivative is given by the global center stable holonomy map $h^{cs}$ for $H^{u}$. Since $\E$ is invariant under $h^{cs}$, this implies that 
\[
(\pi_{y}^{-1} \circ \pi_{x})(\mathscr{A}(p;\E)) = \mathscr{A}(\pi_{y}^{-1}(\pi_{x}(p));\E).
\] 
for any $p \in \wt{W}^{u}(x)$. We thus conclude that for each $\xi \in \p \widetilde{M}$, there is a well defined subset $\mathscr{A}(\xi)$ of $\p \wt{M}$ consisting of all points $\zeta \in \p \wt{M}$ which can be joined to $\xi$ by a curve $\gamma$ in $\p \wt{M}$ which is piecewise $C^{1}$ and tangent to $\E$ in some $\pi_{x}$ coordinate chart (and therefore is tangent to $\E$ in any such coordinate chart). Furthermore the $\Gamma$-equivariance of $\E$ accessibility classes translates into the relation $\gamma (\mathscr{A}(\xi)) =\mathscr{A}( \gamma (\xi)) $.

We would like to show that $\mathscr{A}(\xi)$ is dense in $\p \widetilde{M}$. Let $U$ be an open set in $\p \widetilde{M}$ which does not contain $\xi$. Let $x$ be the image of $\xi$ in an unstable leaf $\wt{W}^{u}(x)$. Take a small open neighborhood $A$ of $x$ which is disjoint from the image of $U$ in $\wt{W}^{u}(x)$. We claim that if $A$ is small enough, then $\mathscr{A}(y;\E)$ intersects the topological boundary $\p A$ of $A$ for every $y \in A$. We begin by reducing this to an equivalent 2-dimensional problem. Take a coordinate chart on $\wt{W}^{u}(x)$ mapping $x$ to the origin of $\R^{m-1}$ and $\E_{x}$ to the coordinate plane corresponding to the first $k$ coordinates, where $k = \dim \E$. Let $C_{\e}$ be the cube $[-\e,\e]^{m-1}$ centered at $p$. Take some $q \in C_{\e}$ and consider the projected image $\bar{q}$ of $q$ in $\R^{k} = \E_{x}$, the first $k$ coordinates. Let $L_{q}$ be the line through $\bar{q}$ parallel to the first coordinate axis. Choose a direction among the last $m-k-1$ coordinates (for definiteness, the ($k+1$)st coordinate). Let $P_{q}$ be the plane spanned by $L_{q}$ and the $(k+1)$st coordinate. As long as $\e$ is small enough (uniformly in $q$), for every $p \in C_{\e}$ the intersection of $\E_{p}$ with $P_{q}$ will be a line. Fix this $\e$ from now on. Identify $P_{q}$ with $\R^{2}$. We see then that it suffices to solve the following equivalent problem: Given an ODE $y' = f(x)$ with $f$ continuous and $|f| \leq K$ everywhere on $[-\e,\e]^{2}$ (note this $K > 0$ is uniform in $q$) show that there is a $C^{1}$ solution $\sigma$ with $\sigma(0) = 0$ such that either $\sigma(t)$ is defined on $[0,\e]$ or else there is some $t \in [0,\e)$ such that $|\sigma(t)| > \e$. 

By the Cauchy-Peano existence theorem for ODEs with continuous coefficients the uniform bound $|f| \leq K$ ensures that there is a uniform $\delta > 0$ such that a solution to the initial value problem $\sigma(t_{0}) = y_{0}$, $\sigma'(t) = f(\sigma(t))$ exists on $[t_{0},t_{0}+\delta]$ provided $|\sigma(t)| \leq \e$ on $[t_{0},t_{0}+\delta]$ and $t \leq \e$ (Theorem 2.19 of \cite{Tes}). Thus, starting at $\sigma(0)=0$, construct a solution existing on $[0,\delta]$, then concatenate this with a solution existing on $[\delta,2\delta]$ and so on. This process ends when either $k \delta > \e$ (which happens after a finite number $\e/\delta$ of steps) or when a solution exceeds $\e$ in absolute value. In either case, we are done. 

Thus $\mathscr{A}(\xi)$ intersects $\p A$ for $A$ small enough. The pairs of endpoints of axes of the isometries $\gamma \in \Gamma$ of $\widetilde{M}$ are dense in $\p \widetilde{M} \times \p \widetilde{M}$, hence we can find an isometry $\gamma$ with the forward endpoint $\gamma_{+} \in U$ of the axis lying in $U$, and the backward endpoint $\gamma_{-} \in A$. Since $\gamma$ gives rise to north-south dynamics on the sphere $\p \widetilde{M}$  there is some $k > 0$ such that $A\subset \gamma^{k}A$ and $\gamma^{k}(\p A) \subset U$. There is thus some $\zeta \in A$ such that $\gamma^{k}\zeta = \xi$. But we know that $\mathscr{A}(\zeta)$ intersects $\p A$ and thus $\gamma^{k}(\mathscr{A}(\zeta)) = \mathscr{A}(\xi)$ intersects $U$. This implies the desired conclusion. 

Fix a periodic point $p \in SM$ of period $T$. Since the bundle $\E$ is a holonomy-invariant subbundle of $H^{u}$, it descends to a subbundle $D\Pi(\E)$ of $TQ^{u}(p)$. $D\Pi(\E)$ is parallel with respect to the connection $\nabla$ constructed in Lemma \ref{connect}, hence since $\nabla$ is torsion-free (as in Lemma \ref{hypirr}), $D\Pi(\E)$ is uniquely integrable and there is thus a foliation $\mathcal{F}$ tangent to $D\Pi(\E)$ inside of $Q^{u}$. Let $\bar{p}$ be the projection of $p$ in $Q^{u}(p)$ and let $\widetilde{\mathcal{F}}(p)$ be the inverse images of all points in the leaf $\mathcal{F}(\bar{p})$ through $\bar{p}$ of $\mathcal{F}$ inside of $Q^{u}(p)$. It is clear that $\mathscr{A}(p;\E) \subset \wt{\mathcal{F}}(p)$, since any piecewise $C^{1}$ curve tangent to $\E$ and passing through $p$ must project to a piecewise $C^{1}$ curve contained entirely inside of $\mathcal{F}(p)$. On the other hand, as shown above, $\mathscr{A}(p;\E)$ must be dense in $W^{u}(p)$.

We thus conclude that $\wt{\mathcal{F}}(p)$ is dense in $W^{u}(p)$, and therefore $\mathcal{F}(\bar{p})$ is dense in $Q^{u}(p)$. But this is absurd unless $\mathcal{F}(\bar{p}) = Q^{u}$: let $U$ be a neighborhood of $\bar{p}$ on which $g^{-T}$ acts as an exponential contraction and such that the foliation $\mathcal{F}$ can be trivialized as slices $\R^{k} \times\{a\}$. Two different slices of the foliation in $U$ have the property that they cannot be connected by a $C^{1}$ curve $\sigma$ lying entirely in $U$. If $\mathcal{F}(\bar{p})$ is dense in $Q^{u}(p)$, we can find some slice of $\mathcal{F}$ in $U$ that does not pass through $\bar{p}$ and a unit speed curve $\sigma:[0,\l] \rightarrow Q^{u}(p)$ with $\sigma(0) = \bar{p}$ and $\sigma(\l) = q \in U$ lying in a different slice. Consider $g^{-kT} \circ \sigma$ for $k > 0$ large. By the definition of the unstable leaf, for $k$ large enough the entire curve $g^{-kT} \circ \sigma$ is contained inside of $U$, and by the $g^{t}$-invariance of $\mathcal{F}$, $g^{-kT} \circ \sigma$ is always contained inside of $\mathcal{F}(\bar{p})$. This implies that $g^{-kT}(\sigma(\l))$ lies in the slice through $\bar{p}$ of $\mathcal{F}$ in $U$. The contraction property of $g^{-T}$ on $U$ implies that $g^{-kT}(U) \subset U$ is an open connected subset containing $\bar{p}$, and the foliation can be trivialized on this open subset, so we can join $g^{-kT}(\sigma(\l))$ to $\bar{p}$ by a curve $\tau$ lying entirely inside of $g^{-kT}(U)$. But then $g^{kT} \circ \tau$ is a curve joining $\bar{p}$ to $\sigma(\l)$ lying entirely inside of the slice of $\mathcal{F}$ through $\bar{p}$ inside of $U$, which is the contradiction that completes the proof.
\end{proof}
We isolate a corollary of the arguments in the proof of Lemma \ref{chypirr} which is of independent interest, obtained by taking $\E = H^{u}$ and ignoring the assumptions on extremal Lyapunov exponents in the above argument. 
\begin{cor}\label{acc}
Let $g^{t}$ be the geodesic flow of a closed negatively curved manifold $M$. Suppose that there is a dominated splitting $E^{u} = H^{u} \oplus V^{u}$, that $E^{u}$ and $H^{u}$ are $\beta$-H\"older continuous, and $Dg^{t}|H^{u}$ is $\beta$-fiber bunched. Then $\mathscr{A}(p;H^{u})$ is dense in $W^{u}(p)$ for every $p \in SM$. 
\end{cor}

\begin{rem}
Given two small disjoint open sets $U$ and $V$ in $W^{u}(p)$, the piecewise $C^{1}$ curve $\gamma$ constructed in Corollary \ref{acc} which starts in $U$ and ends in $V$ will typically take a long, winding route through $W^{u}(p)$ which increases exponentially in length as $U$ and $V$ shrink in size, regardless of how close $U$ and $V$ are in $W^{u}(p)$. Thus it is not immediately clear whether the conclusion in Corollary \ref{acc} can be improved to $\mathscr{A}(p;H^{u}) = W^{u}(p)$. 
\end{rem}

By combining Lemma \ref{amen}, Lemma \ref{confcover}, and Lemma \ref{chypirr}, we see that $Dg^{t}|H^{u}$ preserves a conformal structure which we represent by a Riemannian metric $\tau$ with associated multiplicative cocycle $\psi^{t}$. We will use this conformal structure together with the charts $\{\Psi_{p}\}_{p \in SM}$ given by Assumption (2) of Theorem \ref{weakchyp} to complete the proof of the theorem. Let $r > 0$ given by Assumption (2)  of Theorem \ref{weakchyp}. 

\begin{lem}\label{reg}
Let $p \in SM$, $q \in W^{cs}_{r}(p)$. In the coordinates on $W^{u}_{r}(p)$ and $W^{u}_{r}(q)$ given by $\Psi_{p}$ and $\Psi_{q}$ respectively the center stable holonomy $W^{u}_{r}(p) \rightarrow W^{u}_{r}(q)$ is a projective automorphism of $G_{\mathbb{K}}^{m}$. Consequently the Anosov splitting of $g^{t}$ is $C^{1}$. 
\end{lem}

\begin{proof}
Let $p \in SM$ and $q \in W^{cs}_{r}(p)$. We claim that the center stable holonomy map $\varphi: W^{u}_{r}(p) \rightarrow W^{u}_{r}(q)$ is $C^{1}$. Consider the map $F := \Psi_{q} \circ \varphi \circ \Psi_{p}^{-1}$ defined on a neighborhood of the identity in $G_{\mathbb{K}}^{m}$. By Lemma \ref{wcs} and the assumption that the charts $\Psi_{p}$ map $H^{u}$ to the left-invariant distribution $\mathcal{T}^{m}_{\mathbb{K}}$,  $F$ maps $C^{1}$ curves tangent to $\mathcal{T}^{m}_{\mathbb{K}}$ to $C^{1}$ curves tangent to $\mathcal{H}$. It follows that $F$ is differentiable on the distribution $\mathcal{T}^{m}_{\mathbb{K}}$, with derivative given by the center stable holonomy $h^{cs}$ in these local coordinates. 

This implies that $F$ is Pansu differentiable as a map from $G_{\mathbb{K}}^{m}$ into itself \cite{P2}. The Pansu derivative at each point of $G$ is a homomorphism $DF_{p}^{\mathcal{H}}: G_{\mathbb{K}}^{m} \rightarrow G_{\mathbb{K}}^{m}$ uniquely determined by the derivative action of $F$ on $\mathcal{T}_{\mathbb{K}}^{m}$. Furthermore, since the center stable holonomy preserves the conformal structure $\tau$ on $H^{u}$, this derivative action on $\mathcal{T}_{\mathbb{K}}^{m}$ is conformal in the induced norm of $\tau$ on the horizontal distribution $\mathcal{T}_{\mathbb{K}}^{m}$. This implies that $F$ is 1-quasiconformal in the Carnot-Caratheodory metric on $G_{\mathbb{K}}^{m}$  associated to the metric $\tau$ on the horizontal distribution. But all such 1-quasiconformal maps on any domain in $G$ are given by projective automorphisms of $G_{\mathbb{K}}^{m}$: for $\mathbb{K} = \C$  this is a theorem of Capogna \cite{Ca}, and for $\mathbb{K} = \mathbb{H}$ or $\mathbb{O}$, Pansu showed that any quasiconformal map of $G_{\mathbb{K}}$ into itself is a projective automorphism \cite{P2}. Thus $\Psi_{q} \circ \varphi \circ \Psi_{p}^{-1}$ is smooth and so since $\Psi_{q}$ and $\Psi_{p}$ are $C^{1}$, we conclude that $\varphi$ is $C^{1}$.

This proves that the center stable holonomy between unstable leaves is $C^{1}$. It follows that the center stable bundle $E^{cs}$ is $C^{1}$ along the unstable foliation, and since $E^{s}$ is smooth along the $W^{cs}$ foliation, it follows from Journe's lemma \cite{J} that $E^{s}$ is $C^{1}$. Applying the same reasoning with the roles of $E^{s}$ and $E^{u}$ reversed and using  the flip map $p \rightarrow -p$ for the geodesic flow, we can interchange the role of $E^{s}$ and $E^{u}$ and apply all of the results of this section to $E^{u}$ as well. It follows that $E^{u}$ is $C^{1}$ as well and therefore the Anosov splitting of $g^{t}$ is $C^{1}$. 
\end{proof}

We can find a finite collection of points $p_{1},\dots,p_{k} \in S\wt{M}$ such that the image neighborhoods $U_{i}:=\pi_{p_{i}}(\wt{W}^{u}_{r}(p_{i})) \subset \p \wt{M}$ cover $\p \wt{M}$. For each $U_{i}$ we have a map to the one point compactification $G_{\mathbb{K}}^{m} \cup \{\infty\} = \p H^{m}_{\mathbb{K}}$ given by $\Psi_{p_{i}} \circ \pi_{p_{i}}^{-1}$. By Lemma \ref{reg} the transition maps $\Psi_{p_{j}} \circ \pi_{p_{j}}^{-1} \circ (\Psi_{p_{i}} \circ \pi_{p_{i}}^{-1})^{-1}$ are given by projective automorphisms of $G_{\mathbb{K}}^{m}$. Hence the visual boundary $\p \wt{M}$ of $\wt{M}$ carries a natural $C^{1}$ structure for which $\Gamma = \pi_{1}(M)$ acts by $C^{1}$ maps. In fact, in the coordinates $\Psi_{p_{i}} \circ \pi_{p_{i}}^{-1}$ on $U_{i}$, $\Gamma$ acts by projective automorphisms since $\Gamma$ leaves invariant the lift of the Riemannian metric $\wt{\tau}$ on $H^{u}$ and so we can apply the same rigidity theorems used in the proof of Theorem \ref{reg} for conformal maps on $G_{\mathbb{K}}^{m}$.

We conclude that the natural map $F:\p \wt{M} \rightarrow \p H^{m}_{\mathbb{K}}$ given by the quasi-isometry between the lift $\wt{d}$ of the metric $d$ to $\wt{M}$ and the symmetric metric on $H^{m}_{\mathbb{K}}$ is $C^{1}$ and $\Gamma$-equivariant. Since $F$ is a diffeomorphism it maps the Lebesgue measure class on $\p \wt{M}$ to the Lebesgue measure class on $\p H^{m}_{\mathbb{K}}$. Corollary 4.6 of \cite{Ham2} implies that there is a $C^{1}$ time-preserving (up to scaling) conjugacy between the geodesic flows of $(M,d)$ and $(M,d_{\mathbb{K}})$, and therefore by the minimal entropy rigidity theorem of Besson-Courtois-Gallot \cite{BCG1} the metric $d$ on $M$ is homothetic to the symmetric metric $d_{\mathbb{K}}$. 
\end{section}

\begin{section}{Proof of Theorem 1.4}\label{regproof}
In this final section we assume that $\mathbb{K} = \C$ and that $d$ is a metric in the neighborhood $\mathcal{U}$ of $d_{\C}$ described at the beginning of Section \ref{thirdthm}. We assume that there is a $g^{t}$-invariant ergodic, fully supported probability measure $\mu$ on $SM$ with local product structure such that $\la_{+}(Dg^{t}|H^{u},\mu) = \la_{-}(Dg^{t}|H^{u},\mu)$.  We will show that if $H^{u}$ has $C^{1}$ regularity along $W^{u}$ leaves with H\"older continuous derivatives then there is an $r > 0$ such that we can find $C^{1}$ charts $\Psi_{p}: W^{u}_{r}(p) \rightarrow G_{\C}^{m}$ which map $H^{u}$ to $\mathcal{T}_{\C}^{m}$. Hence by Theorem \ref{weakchyp} we conclude that $(M,d)$ is homothetic to $(M,d_{\mathbb{K}})$. 

We first note that the $C^{1}$ regularity of $H^{u}$ along $W^{u}$ implies $C^{1}$ regularity for the unstable holonomies of $W^{u}$, 

\begin{lem}\label{reghor}
On a fixed unstable leaf $W^{u}(p)$, the unstable holonomy $(x,y) \rightarrow h^{u}_{xy}$, $x,y \in W^{u}(p)$ is $C^{1}$ as a function of $x$ and $y$. The derivative of the unstable holonomy is uniformly H\"older continuous in $x$ and $y$.  
\end{lem}

\begin{proof}
Since $H^{u}$ is $C^{1}$ along $W^{u}$ leaves and $E^{u}/V^{u}$ is smooth along these leaves, the projection $H^{u} \rightarrow E^{u}/V^{u}$ is $C^{1}$. As remarked in Lemma \ref{connect}, the unstable holonomy of the action of $Dg^{t}$ on $E^{u}/V^{u}$ is induced by the unstable holonomy of $Dg^{t}|H^{u}$ by this projection. Since the unstable holonomy on $E^{u}/V^{u}$ is $C^{1}$, this implies that the unstable holonomy for $H^{u}$ is $C^{1}$. 
\end{proof}

By applying the arguments of Section \ref{thirdthm} prior to Lemma \ref{reg} which do not use the existence of the charts $\Psi_{p}$, we conclude that there is a Riemannian metric $\tau$ on $H^{u}$ whose conformal class is invariant under stable and unstable holonomies as well as the action of $Dg^{t}$. By Lemma \ref{reghor}, we can choose a representative of the conformal class of $\tau$ which is $C^{1}$ along $W^{u}$. We will fix this representative from now on. We let $\psi^{t}$ be the multiplicative $\R$-cocycle such that $(g^{t})^{*}\tau = \psi^{t} \tau$. 

We claim we may assume that the line bundle $V^{u}$ is orientable. If $V^{u}$ is not orientable, pass to a double cover of $SM$ in which $V^{u}$ is orientable. This corresponds to the geodesic flow on a Riemannian double cover $M^{*}$ of $M$. If we prove that $M^{*}$ is isometric to a complex hyperbolic manifold, it then follows immediately that $M$ is as well. Hence it suffices to assume $V^{u}$ is orientable.

Extend $\tau$ to a Riemannian metric on $E^{u}$ by setting $V^{u}$ and $H^{u}$ to be orthogonal and taking some continuous section $Z$ of $V^{u}$ which is smooth along $W^{u}$ leaves to define the unit length on $V^{u}$. $\tau$ remains $C^{1}$ since $H^{u}$ is a $C^{1}$ subbundle of $E^{u}$. Let $\alpha$ be the 1-form on $W^{u}$ defined by taking the inner product using $\tau$ with $Z$. It's clear that $\ker \alpha \cap E^{u} = H^{u}$ and that $\alpha$ is $C^{1}$. Since $V^{u}$ and $H^{u}$ are $Dg^{t}$-invariant, for each $t \in \R$ there is some $C^{1}$ (along $W^{u}$ leaves) function $\varphi^{t}$ such that $(g^{t})^{*}\alpha = \varphi^{t}\alpha$. $\varphi^{t}$ is a multiplicative $\R$-cocycle.

\begin{lem}\label{nondeg}
$d\alpha|H^{u}$ is nondegenerate on all of $SM$.  
\end{lem}

\begin{proof}
Observe first that for every $t \in \R$,
\[
(g^{t})^{*}(d\alpha) = d((g^{t})^{*}(\alpha)) = d(\varphi^{t}\alpha) = d\varphi^{t} \wedge \alpha + \varphi^{t}d\alpha
\]
and hence $(g^{t})^{*}d\alpha|H^{u} = \varphi^{t}d\alpha|H^{u}$. It follows that 
\[
R_{k} = \{p \in SM: \; \dim \{v \in H^{u}_{p}: \; \iota_{v}d\alpha = 0\} = k\}
\]
is a $g^{t}$-invariant subset of $SM$. Since $\mu$ is ergodic with respect to the action of $g^{t}$ and $\bigcup_{k=0}^{m-2}R_{k} = SM$, $\mu(R_{k}) = 1$ for some $k$. If $0 < k < m-2$, then $p \rightarrow \{v \in H^{u}_{p}: \; \iota_{v}d\alpha = 0\}$ is a proper measurable $g^{t}$ invariant subbundle of $H^{u}$, which is impossible by Lemma \ref{chypirr}. 

Suppose now that $\mu(R_{0}) = 1$. Then $d\alpha|H^{u} = 0$ for a dense set of points in $SM$ since $\mu$ has full support in $SM$, hence $d\alpha|H^{u} = 0$ for every $p \in SM$ since $d\alpha$ is continuous. This implies by the $C^{1}$ Frobenius theorem that $H^{u}$ is uniquely integrable as a subbundle of $E^{u}$ over $W^{u}$. We claim that this contradicts Corollary \ref{acc}. Fix a periodic point $p \in SM$ of period $T$ and let $\mathcal{F}$ be the $C^{1}$ foliation tangent to $H^{u}$ inside $W^{u}(p)$. By Corollary \ref{acc}, each leaf of $\mathcal{F}$ is dense in $W^{u}(p)$. 

But this is absurd, for reasons similar to those given in the proof of Lemma \ref{chypirr}. Let $U$ be a neighborhood of $p$ on which $g^{-T}$ acts as an exponential contraction and such that the foliation $\mathcal{F}$ can be trivialized as plaques $\R^{m-2} \times\{a\}$. Two different plaques of the foliation in $U$ have the property that they cannot be connected by a $C^{1}$ curve $\sigma$ lying entirely in $U$. If $\mathcal{F}(p)$ is dense in $W^{u}(p)$, we can find some slice of $\mathcal{F}$ in $U$ that does not pass through $p$ and a unit speed curve $\gamma:[0,\l] \rightarrow W^{u}(p)$ with $\gamma(0) = p$ and $\gamma(\l) = q \in U$ lying in a different plaque. Consider $g^{-kT} \circ \gamma$ for $k > 0$ large. By the definition of the unstable leaf, for $k$ large enough the entire curve $g^{-kT} \circ \sigma$ is contained inside of $U$, and by the $g^{T}$-invariance of $\mathcal{F}$, $g^{-kT} \circ \sigma$ is always contained inside of $\mathcal{F}(p)$. This implies that $g^{-kT}(\sigma(\l))$ lies in the local leaf through $p$ of $\mathcal{F}$ in $U$. The contraction property of $g^{-T}$ on $U$ implies that $g^{-kT}(U) \subset U$ is an open connected subset containing $p$, and the foliation can be trivialized on this open subset, so we can join $g^{-kT}(\sigma(\l))$ to $p$ by a curve $\sigma$ lying entirely inside of $g^{-kT}(U)$. But then $g^{kT} \circ \sigma$ is a curve joining $p$ to $\gamma(\l)$ lying entirely inside of the local leaf of $\mathcal{F}$ through $p$ inside of $U$, which is a contradiction.

The only remaining possibility is $\mu(R_{m-2}) = 1$. Hence $d\alpha_{p}|H_{p}^{u}$ is nondegenerate for $\mu$-a.e. $p \in SM$. We conclude that $m$ is even. Let $\zeta_{p}$ denote the volume element on $H^{u}_{p}$ induced by the inner product $\tau$ on $H^{u}_{P}$. Then since $d\alpha|H^{u}$ is nondegenerate $\mu$-a.e. and the vector space $\bigwedge^{m-2}H^{u}$ is 1-dimensional, there is a continuous function $F:SM \rightarrow \R$ such that 
\[
F \cdot \zeta = d\alpha^{m/2}
\]
and $F$ is nonzero on a set of full $\mu$-measure. Note that $(g^{t})^{*}\zeta = (\psi^{t})^{\frac{m}{2}}\zeta$ and therefore we conclude upon applying $(g^{t})^{*}$ to both sides of the above equation that 
\[
F(g^{t}(p)) (\psi^{t}(p))^{\frac{m}{2}} \zeta_{p} = (\varphi^{t}(p))^{\frac{m}{2}} d\alpha^{\frac{m}{2}}_{p} = F(p) (\varphi^{t}(p))^{\frac{m}{2}} \zeta_{p}
\]
and therefore at every $p \in SM$ for which $F(p) \neq 0$ and every $t \in \R$, 
\[
\frac{2}{m}\left(\log(|F|(g^{t}(p)))- \log(|F|(p))\right) = \log(\varphi^{t}(p)) - \log(\psi^{t}(p))
\]
This equation holds on a set of full $\mu$-measure (in fact on an open and dense subset of full volume in $SM$) and thus implies that $\frac{2}{m}\log |F|$ is a measurable solution to the cohomological equation for the cocycle $\frac{\varphi^{t}(p)}{\psi^{t}(p)}$. Since $\varphi^{t}$ and $\psi^{t}$ are H\"older continuous cocycles, this implies by the measurable rigidity of the cohomological equation over Anosov flows \cite{Wal} that $\frac{2}{m}\log |F|$ coincides $\mu$-a.e. with a continuous solution of the cohomological equation. But since $\frac{2}{m}\log |F|$ is continuous on the set $R_{0}$ of points at which $d\alpha$ is nondegenerate, and this is an open and dense subset of $SM$, this then implies that $\log |F|$ extends to a continuous function on $SM$. It follows that $|F| > 0$ on $SM$ and therefore that $d\alpha|H^{u}$ is nondegenerate on all of $SM$. 
\end{proof}

Put the standard coordinates $(x_{1},\dots,x_{k},y_{1},\dots,y_{k},z)$ on $\R^{m-1}$, where $k = (m-2)/2$. Let
\[
\nu = dz + \frac{1}{2}\sum_{i=1}^{k}x_{i}dy_{i} - y_{i}dx_{i}
\]
be the standard contact 1-form on $\R^{m-1}$. If we identify $\R^{m-1}$ with the Heisenberg group $G^{m}_{\C}$, then $\ker \nu$ gives the left-invariant distribution $\mathcal{T}^{m}_{\C}$. The following final lemma constructs the desired charts $\{\Psi_{p}\}_{p \in SM}$ which completes the proof of Theorem \ref{regchyp}.

\begin{lem}\label{chart}
For each $p \in SM$, there is an $r > 0$ such that there is a $C^{1}$ function $\xi_{p}: W^{u}_{r}(p) \rightarrow (0,\infty)$ and a $C^{2}$ diffeomorphism $\Psi_{p}: W^{u}_{r}(p) \rightarrow U$, $U$ a neighborhood of the identity in $\R^{m-1}$, with $\Psi_{p}^{*}(\nu) = \xi_{p}\alpha$. 
\end{lem}

\begin{proof}
The line bundle spanned by $d\alpha|H^{u}$ inside of $\Lambda^{2}H^{u}$ is invariant under $g^{t}$ and therefore by Lemma \ref{contsub} is invariant under unstable holonomy. Hence for every point $p \in SM$ there is a $C^{1}$ function $\xi_{p}: W^{u}(p) \rightarrow (0,\infty)$ defined by
\[
(h^{u}_{pq})^{*}d\alpha = \xi_{p}(q)d\alpha
\]
where $d\alpha$ is restricted to $H^{u}$ (here we again use Lemma \ref{reghor} for the $C^{1}$ regularity assertion). The holonomy relation $Dg^{t}_{q} \circ h^{u}_{pq} =  h^{u}_{g^{t}p g^{t}q} \circ Dg^{t}_{q}$ implies that $\varphi^{t}(q)\xi_{p}(q) = \xi_{g^{t}(p)}(g^{t}(q))\varphi^{t}(p)$, since $d\alpha|H^{u}$ scales by $\varphi^{t}$ when acted on by $Dg^{t}$. 

Consider the $C^{1}$ 1-form $\xi_{p}\alpha$ on $W^{u}_{r}(p)$. For $q \in W^{u}(g^{-t}p)$, $X \in E^{u}_{q}$, 
\[
\xi_{p}(g^{t}(q))\alpha(Dg^{t}X) = \xi_{p}(g^{t}(q))\varphi^{t}(q)\alpha(X) = \xi_{g^{-t}p}(q)\varphi^{t}(g^{-t}(p))\alpha(X)
\]
Hence $(g^{t})^{*}(\xi_{p}\alpha) = \varphi^{t}(g^{-t}(p))\xi_{g^{-t}(p)}\alpha$. Taking the exterior derivative of each side, we get
\[
(g^{t})^{*}(d(\xi_{p}\alpha)) = \varphi^{t}(g^{-t}(p))d(\xi_{g^{-t}(p)}\alpha)
\]
Let $Z_{p}$ be the Reeb vector field for $\xi_{p}\alpha$ on $W^{u}(p)$ defined by 
\[
\xi_{p}\alpha(Z_{p}) = 1
\] 
\[
d(\xi_{p}\alpha)(Z_{p},X) = 0\; \text{for every $X \in H^{u}$}.
\]
From the exterior derivative relation we obtain
\[
Dg^{t}(Z_{p}) = \varphi^{t}(p)Z_{g^{t}(p)},
\]
and thus we can write
\[
Z_{p} = \varphi^{-t}(p)Dg^{t}(Z_{g^{-t}(p)}).
\]
We claim that there is a small $\delta > 0$ such that there is a constant $\gamma > 0$ independent of $p$ such that $Z_{p}$ makes an angle of at least $\gamma$ with $H^{u}$ on $W^{u}_{\delta}(p)$.  Note that for every $p \in SM$ we have $Z_{p} \notin H^{u}$ since $d(\xi_{p}\alpha)|H^{u} = d\alpha|H^{u}$ is nondegenerate.

We first show that there is some $\gamma > 0$ such that $Z_{p}(p)$ makes an angle of at least $2 \gamma$ with $H^{u}$. If this did not hold, then we could find a sequence of points $p_{n} \in SM$ converging to $p \in SM$ with $Z_{p_{n}}/\|Z_{p_{n}}\| \rightarrow Y \in H^{u}_{p}$. Let $\delta > 0$ be small enough that the unstable holonomy $h^{u}$ of $H^{u}$ has uniformly H\"older continuous derivative on $W^{u}_{\delta}(q)$ for every $q \in SM$. Then the unstable holonomies and their derivatives on $W^{u}_{\delta}(p_{n})$ converge uniformly to the unstable holonomy and its derivative on $W^{u}_{\delta}(p)$. It follows that the $C^{1}$ 1-forms $\xi_{p_{n}}\alpha$ on $W^{u}_{\delta}(p_{n})$ converge uniformly in the $C^{1}$ topology to the 1-form $\xi_{p}\alpha$ on $W^{u}_{\delta}(p)$. In particular $d(\xi_{p_{n}}\alpha)$ converges uniformly to $d(\xi_{p}\alpha)$. Hence $Z_{p_{n}} \rightarrow Z_{p}$ uniformly in $n$. But the assumption that $Z_{p_{n}}/\|Z_{p_{n}}\| \rightarrow Y \in H^{u}_{p}$ implies that $Z_{p} \in H^{u}_{p}$, which is impossible. 

We first show that there is some $\gamma > 0$ such that $Z_{p}(p)$ makes an angle of at least $2 \gamma$ with $H^{u}$. Hence, since the unstable holonomy has uniformly H\"older derivative on $W^{u}_{\delta}(p)$, there is some $r > 0$ independent of $p$ such that $Z_{p}$ makes an angle of at least $\gamma$ with $H^{u}$ on $W^{u}_{r}(p)$. Since for $t > 0$ large enough $g^{-t}(W^{u}_{r}(p)) \subset W^{u}_{r}(g^{-t}p)$, $Z_{g^{-t}(p)}$ makes an angle of at least $\gamma$ with $H^{u}$ on $g^{-t}(W^{u}_{r}(p))$. Thus the expression 
\[
Z_{p} = \varphi^{-t}(p)Dg^{t}(Z_{g^{-t}(p)})
\] 
implies that $Z_{p}$ is parallel to $V^{u}$, since as $t \rightarrow \infty$ the angle between $V^{u}$ and $Dg^{t}(Z_{g^{-t}(p)})$ converges uniformly to zero on $W^{u}_{r}(p)$. Here we need that the angle of $Z_{g^{-t}(p)}$ with $H^{u}$ is uniformly bounded away from zero on $g^{-t}(W^{u}_{r}(p))$.

Thus $Z_{p}$ is parallel to $V^{u}$ and so $d(\xi_{p}\alpha)$ vanishes on $V^{u}$. Since $d(\xi_{p}\alpha)|H^{u} = \xi_{p}d\alpha$, we conclude that $d(\xi_{p}\alpha)|H^{u}$ is invariant under unstable holonomy. In particular it is $C^{1}$. Since the splitting $E^{u} = H^{u} \oplus V^{u}$ is $C^{1}$, we get that $d(\xi_{p}\alpha)$ itself is $C^{1}$. 

We conclude that $\xi_{p}\alpha$ is a $C^{1}$ contact form on $W^{u}_{loc}(p)$ with $d(\xi_{p}\alpha)$ also being $C^{1}$. The foliation associated to the Reeb vector field $Z_{p}$ is the strong unstable foliation $W^{vu}$, which is smooth. $d(\xi_{p}\alpha)$ descends to a $C^{1}$ closed 2-form $\omega$ on the quotient $Q^{u}(p)$ of $W^{u}$ by the $W^{vu}$ foliation. 

Recall that $\Pi: W^{u}(p) \rightarrow Q^{u}(p)$ denotes projection. Darboux's theorem for symplectic forms implies that there is a neighborhood $U$ of $\Pi(p)$ and a $C^{2}$ diffeomorphism $F: U \rightarrow L$ onto a neighborhood $L \subset \R^{m-2}$ of 0 which pulls back the standard symplectic form $d\nu$ on $L$ to $\omega$ on $U$. Choose a smooth transversal $K$ to the $W^{vu}$ foliation of $W^{u}(p)$ with $\Pi(K) = U$ and $p \in K$. Let $\wt{U}$ be the open neighborhood of $p$ given by taking, for each $q \in K$, an open neighborhood of $q$ in $W^{vu}(q)$. Put coordinates on $\wt{U}$ by smoothly identifying $K$ with a neighborhood of 0 in $\R^{m-2}$, and smoothly trivializing the $W^{vu}$ foliation to the map to the $z$ coordinate axis (the last coordinate) in a neighborhood of 0 in $\R^{m-1}$. 

Using these coordinates we lift $F$ to a $C^{1}$ diffeomorphism $\Psi_{p}: \wt{U} \rightarrow \wt{L}$ of a neighborhood $\wt{U}$ of $p \in W^{u}(p)$ to a neighborhood $\wt{L}$ of 0 in $\R^{m-1}$. Since a contact form is invariant under its Reeb flow, the coordinate representation of $\xi_{p}\alpha$ in the above defined coordinates on $\wt{U}$ is constant along the $z$-axis. It follows that $d(\xi_{p}\alpha)$ is also constant along the $z$-axis. Define $\phi: \wt{U} \rightarrow \wt{L}$ in these coordinates by $\phi(x,z) = (F(x),z)$, where $x \in U$. Then $\phi^{*}(d\nu) = d(\xi_{p}\alpha)$ and therefore the 1-form $\beta = \phi^{*}(\nu) -\xi_{p}\alpha$ is closed. Define for $x \in U$,
\[
f(x) =\int_{0}^{1}\beta(tx) \, dt
\] 
Note that the $z$-component of $\beta$ is 0 since $\xi_{p}\alpha$ is independent of the $z$-coordinate and $\phi$ is the identity on the $z$-coordinate. Note also that $f$ does not depend on $z$. 

Set $\Psi_{p}(x,z) = (\phi(x),z-f(x))$. We claim that $\Psi_{p}^{*}(\nu) = \xi_{p}\alpha$. Let $\gamma: [0,1] \rightarrow \wt{U}$ and let $\sigma_{t}$, $t \in [0,1]$, be the unique radial curve (when projected to $U$) tangent to $\ker \alpha$ joining a point on the $z$-axis to $\gamma(1)$. We thus obtain a continuous map $(t,s) \rightarrow \sigma_{t}(s)$ of $[0,1]^{2}$ into $\wt{U}$, which we will denote by $\sigma$.

Now consider the lift of a radial curve $\eta: s \rightarrow sx$ ($x \in U$) to a curve $\tilde{\eta}$  tangent to $\ker \alpha$. The $z$-coordinate of $\tilde{\eta}$ is then given by $\int_{0}^{t}-\xi_{p}(sx)\alpha(sx) \, ds$. Then 
\[
\Psi(\tilde{\eta}(t)) = \left(\phi(tx), \int_{0}^{t}-\xi_{p}(sx)\alpha(sx) \, ds-f(tx)\right) = \left(\phi(tx), \int_{0}^{t}-\phi^{*}\nu(sx) \, ds\right)
\]
which is the lift of the curve $s \rightarrow \phi(sx)$ in $B(r)$ to a curve tangent to $\ker \nu$. We thus conclude that $\Psi$ maps lifts of radial curves tangent to $\ker \alpha$ to curves tangent to $\ker \nu$. We also note that we still have $\Psi_{p}^{*}(d\nu) = d(\xi_{p}\alpha)$: both $d\nu$ and $d(\xi_{p}\alpha)$ vanish on vectors parallel to the $z$-axis, and $D\Psi_{p}$ has the same action as $D\phi$ on the components of the vectors lying in $U$. As a consequence, the 1-form $\kappa = \Psi_{p}^{*}(\nu) - \xi_{p}\alpha$ is closed. This implies that the path integral of $\sigma^{*}\kappa$ around the boundary of $[0,1]^{2}$ is zero. We just showed that on radial curves, $\ker \Psi_{p}^{*}\nu = \ker \alpha$ , so that the path integral of $\kappa$ over $\sigma_{0}$ and $\sigma_{1}$ is zero. The curve $t \rightarrow \sigma_{t}(0)$ is tangent to the $z$-axis, and $\Psi^{*}(\nu) - \xi_{p}\alpha$ vanishes on the $z$-axis, since $\Psi$ restricts to the identity on the $z$-axis and $\xi_{p}\alpha(\p_{z}) = \nu(\p_{z}) = 1$, where $\p_{z}$ is the coordinate vector field parallel to the $z$-axis. Thus we conclude that 
\[
\int_{\gamma} \kappa = 0
\]
Since this holds for every curve $\gamma$ in $\wt{U}$, we conclude that $\kappa = 0$, so that $\Psi_{p}^{*}(\nu) = \xi_{p}\alpha$.
\end{proof} 
\end{section}
\bibliographystyle{plain}
\bibliography{Clarkbibtex}
\end{document}